\documentclass[12pt,a4paper]{amsart}

\usepackage{graphics,epic}
\usepackage{amsmath,amssymb, amsthm}
\usepackage[all,2cell]{xy}

\newtheorem{theorem}{Theorem}[section]
\newtheorem*{theorem*}{Theorem}

\newtheorem{lemma}[theorem]{Lemma}
\newtheorem{proposition}[theorem]{Proposition}

\newtheorem*{conjecture*}{Conjecture}

\newtheorem{example}[theorem]{Example}
\newtheorem{remark}[theorem]{Remark}

\newcommand{\ie}{{\em i.e.}\ }
\newcommand{\confer}{{\em cf.}\ }
\newcommand{\eg}{{\em e.g.}\ }

\newcommand{\Tot}{\mathrm{Tot}}

\newcommand{\opname}[1]{\operatorname{\mathsf{#1}}}

\newcommand{\grmod}{\opname{grmod}\nolimits}

\newcommand{\Grmod}{\opname{Grmod}\nolimits}
\newcommand{\per}{\opname{per}\nolimits}

\newcommand{\Gr}{\opname{Gr}\nolimits}
\newcommand{\dimv}{\underline{\dim}\,}

\newcommand{\Dif}{\opname{Dif}\nolimits}
\newcommand{\ind}{\opname{ind}}

\newcommand{\rep}{\opname{rep}\nolimits}
\newcommand{\Rep}{\opname{Rep}\nolimits}

\newcommand{\cok}{\opname{cok}\nolimits}

\renewcommand{\ker}{\opname{ker}\nolimits}

\newcommand{\Z}{\mathbb{Z}}
\newcommand{\N}{\mathbb{N}}
\newcommand{\Q}{\mathbb{Q}}

\renewcommand{\P}{\mathbb{P}}

\newcommand{\id}{\mathbf{1}}

\newcommand{\Hom}{\opname{Hom}}
\newcommand{\go}{\opname{G_0}}

\newcommand{\Aut}{\opname{Aut}}
\newcommand{\End}{\opname{End}}

\newcommand{\ten}{\otimes}

\newcommand{\ca}{{\mathcal A}}
\newcommand{\cb}{{\mathcal B}}
\newcommand{\cc}{{\mathcal C}}
\newcommand{\cd}{{\mathcal D}}

\newcommand{\ch}{{\mathcal H}}

\newcommand{\cm}{{\mathcal M}}

\newcommand{\cs}{{\mathcal S}}
\newcommand{\ct}{{\mathcal T}}

\newcommand{\mh}{\mathfrak{h}}
\newcommand{\mg}{\mathfrak{g}}
\newcommand{\mn}{\mathfrak{n}}
\newcommand{\n}{\mathfrak{N}}

\begin{document}

\title{The Ringel--Hall Lie algebra of a spherical object}

\author{Changjian Fu}
\address{Changjian Fu\\Department of Mathematics\\SiChuan University\\610064 Chengdu\\P.R.China}
\email{changjianfu@scu.edu.cn}

\author{Dong Yang}
\address{
Dong Yang\\Hausdorff Research Institute for
Mathematics\\ Poppelsdorfer Allee 45\\ D-53115 Bonn\\
Germany} \email{dongyang2002@gmail.com}

\date{December 2008. Last modified on \today}

\begin{abstract} For an integer $w$, let $\cs_w$ be the algebraic triangulated category
generated by a $w$-spherical object. We determine the Picard group
of $\cs_w$ and show that each orbit category of $\cs_w$ is
triangulated and is triangle equivalent to a certain orbit category
of the bounded derived category of a standard tube. When $n=2$, the
orbit category $\cs_w/\Sigma^2$ is 2-periodic triangulated, and we
characterize the
associated Ringel--Hall Lie algebra in the sense of Peng and Xiao.\\
{MSC classification 2010:} 17B99, 18E30, 16E35, 16E45.
\end{abstract}
\maketitle

\section{Introduction}\label{s:introduction}


The Hall algebra of an algebraic triangulated category (with certain
finiteness conditions) was introduced by To\"en~\cite{Toen06}
(\confer also~\cite{XiaoXu2008}~\cite{ChenXu10}), with the
expectation to realize the quantized/universal enveloping algebra of
a Kac--Moody Lie algebra $\mathfrak{g}$ from the bounded derived
category $\cd_Q$ of finite-dimensional representations over a quiver
$Q$ whose underlying graph is the Dynkin graph of $\mathfrak{g}$.
To\"en was inspired by the work of
Peng--Xiao~\cite{PengXiao97}~\cite{Pengxiao1998}~\cite{PengXiao2000},
in which the authors provided a way to construct Lie algebras
(called \emph{Ringel--Hall Lie algebras}) from 2-periodic
triangulated categories, \eg to construct the Kac--Moody Lie algebra
$\mathfrak{g}$ from the root category $\cd_Q/\Sigma^2$, where
$\Sigma$ is the suspension functor of $\cd_Q$, \confer
also~\cite{XiaoXuZhang06}~\cite{Fu10}.

For an integer $w$, let $\cs_w$ be the algebraic triangulated
category over a field $k$ generated by a $w$-spherical object.
In~\cite{KellerYangZhou09}, the Hall algebra of $\cs_w$ is computed.
Having Peng--Xiao's and To\"en's work in mind, it is natural to ask
what is the Ringel--Hall Lie algebra of the orbit category
$\cs_w/\Sigma^2$.

In order to apply Peng--Xiao's construction, we need to first prove
that the orbit category $\cs_w/\Sigma^2$ is triangulated. In fact,
we will prove that each reasonable orbit category of $\cs_w$ carries
a nice triangle structure. Moreover, we will provide a
characterization of these categories in terms of orbit categories of
bounded derived categories of standard tubes. More precisely, we
have
\begin{theorem}\label{T:triangle-orbit}
Let $n$ be a positive integer.
\begin{itemize}
\item[(a)] There are group isomorphisms $f:\Aut(\cs_w)\rightarrow k^{\times}\times\mathbb{Z}$ and $g:\Aut(\cd^b(\ct_n))\rightarrow k^{\times}\times\mathbb{Z}/n\mathbb{Z}\times\mathbb{Z}$. Here $\ct_n$ is the standard tube of rank $n$ and $\cd^b(\ct_n)$ is the bounded derived category. In particular, $f(\Sigma)=(1,1)$, $g(\tau)=(1,1,0)$ and $g(\Sigma)=(1,0,1)$, where $\tau$ is the Auslander--Reiten translation of $\ct_n$.
\item[(b)] For any $a\in k^{\times}$, the orbit category $\cs_w /f^{-1}(a,n)$ carries a canonical
triangle structure such that the projection functor
$\cs_w\rightarrow\cs_w/f^{-1}(a,n)$ is a triangle functor. Moreover, it
is triangle equivalent to the perfect derived category of the dg
algebra $\tilde{\Lambda}=k\langle
s,r,r^{-1}\rangle/(s^2,sr=(-1)^{nw}ars)$ with trivial differential,
where $\deg(s)=w$ and $\deg(r)=n$.
\item[(c)] Let $m$ be the greatest
common divisor of $n$ and $d=1-w$, $n'=n/m$, $d'=d/m$, $c$ be an inverse
of $d'$ modulo $n'$, and let $a,b\in k^{\times}$ be such that $a=((-1)^d b)^{n'}$. Then the orbit category $\cs_w/f^{-1}(a,n)$ is triangle equivalent
to the orbit category $\cd^b(\ct_{n'})/g^{-1}(b,c,m)$ (which admits
a canonical triangle structure by~\cite[Theorem 9.9]{Keller05}).

\end{itemize}
\end{theorem}

In particular, the orbit category $\cs_w/\Sigma^2$ is
triangle equivalent to the root category of the standard homogeneous tube
(\ie the standard tube of rank 1) if $w$ is odd, and to the cluster tube of
rank 2 if $w$ is even. This characterization helps us to obtain the
following description of the associated Ringel--Hall Lie algebra in
terms of a basis and the corresponding structure constants. 

\begin{theorem}\label{t:rh-lie-alg} Let $\mathfrak{g}$ denote the Ringel--Hall Lie
algebra of $\cs_w/\Sigma^2$ with scalar extended to $\mathbb{Q}$.
\begin{itemize}
\item[(a)] If $w$ is odd, then $\mathfrak{g}$ is isomorphic to the
infinite-dimensional Heisenberg Lie algebra.
\item[(b)] If $w$ is even, then $\mathfrak{g}$ has a large center. The quotient of $\mathfrak{g}$ by its center
has a basis $\{a_x|x\in\mathbb{N}\cup\{0\}\}\cup \{b_y,c_y|y\in\mathbb{N}-\frac{1}{2}\}$ and the structure
constants are
\begin{itemize}
\item[$\cdot$] $[a_x,a_{x'}]=0$, $[b_y,b_{y'}]=0$, $[c_y,c_{y'}]=0$;
\item[$\cdot$] $[a_x,b_y]=b_{y+x}+\mathrm{sgn}(y-x)b_{|y-x|}$,
$[a_x,c_y]=-c_{y+x}-\mathrm{sgn}(y-x)c_{|y-x|}$;
\item[$\cdot$] $[b_y,c_{y'}]=2a_{y+y'}-2a_{|y-y'|}$.
\end{itemize}
where for an integer $r$, $\mathrm{sgn}(r)=1$ if $r$ is positive and $\mathrm{sgn}(r)=-1$ if $r$ is negative.
\end{itemize}
\end{theorem}
We remark that the cluster tube of rank $2$ is not proper, and we
refer to \cite{XiaoXuZhang06} for Lie algebras constructed from
non-proper $2$-periodic triangulated categories ({\it cf.}  also
Section \ref{s:Ringel-Hall}). The Lie algebra obtained in
Theorem~\ref{t:rh-lie-alg} (b) seems likely to be  the first Lie
algebra computed from a non-proper 2-periodic triangulated category.

This paper is organized as follows. In
Section~\ref{s:preliminaries}, we give some preliminary results,
including results on derived categories of dg categories,
triangulated orbit categories, the algebraic triangulated category
$\cs_w$ generated by a $w$-spherical object and the bounded derived
category of a standard tube. In particular, we prove part (a) of
Theorem~\ref{T:triangle-orbit}. In
Section~\ref{s:triangle-structure}, we prove
Theorem~\ref{T:triangle-orbit} (b) and (c). The proof of the first
part of (b) is a variant of Keller's proof of~\cite[Theorem
4]{Keller05}. In the proof of the second part of (b) and the part
(c), we compute the dg algebras for both orbit categories and
compare them. Section~\ref{s:Ringel-Hall} is devoted to the
characterization (Theorem~\ref{t:rh-lie-alg}) of the Ringel--Hall
Lie algebras associated to $\cs_w/\Sigma^2$. In an appendix, by
using covering and the universal property of orbit categories we
construct an explicit triangle equivalences between the two orbit
categories in Theorem~\ref{T:triangle-orbit} (c) in the case when
$n$ is even and $a=b=1$.


The second-named author gratefully acknowledges support from
Max-Planck-Institut f\"ur Mathematik in Bonn and from Hausdorff
Research Institute for Mathematics. He thanks Dong Liu for helpful
conversations. Both authors are deeply grateful to Bernhard Keller
for pointing out to them the similarity between the orbit category
$\cs_{3}/\Sigma^2$ and the root category of a homogeneous tube
(which inspired the study in the second part of
Section~\ref{s:triangle-structure}) and for his great help in
finding an error in a previous version. They thank Martin Kalck for
pointing out some typos to them.

\section{Preliminaries}\label{s:preliminaries}

Let $k$ be a field.


\subsection{The derived category of a dg category}\label{ss:derived-category} We
follow~\cite{Keller94}. Let $\ca$ be a differential graded (=dg)
$k$-category (we identify a dg $k$-algebra with a dg $k$-category
with one object). Let $\Dif\ca$ be the dg category of (right) dg modules over $\ca$. For two dg modules $M$ and $N$ over $\ca$ and for an integer $n$, the degree $n$ component of the morphism complex $\Hom_{\Dif\ca}(M,N)$ consists of the homogeneous morphisms from $M$ to $N$ of degree $n$, here $M$ and $N$ are considered as graded modules over the underlying graded category of $\ca$. The differential of $\Hom_{\Dif\ca}(M,N)$ is induced from the differentials of $M$ and $N$. The shift of complexes is a dg functor $\Sigma:\Dif\ca\rightarrow\Dif\ca$, which takes a homogeneous morphism $f$ of degree $n$ to $(-1)^n f$.

The derived category $\cd\ca$ of $\ca$ has the same objects as $\Dif\ca$ and its morphisms are
obtained from the closed morphisms in $\Dif\ca$ of degree $0$ by formally inverting all
quasi-isomorphisms. It is triangulated with suspension functor
$\Sigma$ the shift functor. Let $\per\ca$ denote the smallest
triangulated subcategory of $\cd\ca$ containing all free dg
$\ca$-modules and closed under taking direct summands. Let
$\cd_{fd}\ca$ be the full subcategory of $\cd\ca$ consisting of
those dg modules which has finite-dimensional total cohomology. It
is a triangulated subcategory of $\cd\ca$.


Let $\ca$ be a dg $k$-category. We define $H^0\ca$ to be the
$k$-category which has the same objects as $\ca$ and whose morphism
space $\Hom_{H^0\ca}(X,Y)$ between two objects $X$ and $Y$ is the
zeroth cohomology of the complex $\Hom_{\ca}(X,Y)$. The Yoneda embedding
$\ca\hookrightarrow\Dif\ca$ of dg categories induces an embedding
$H^0\ca\hookrightarrow\cd\ca$ of $k$-categories. In particular, we have for $A\in\ca$ and the corresponding  free module $A^\wedge=\Hom_{\ca}(?,A)$
\[H^*\Hom_{\ca}(A,A)=\bigoplus_{p\in\mathbb{Z}}\Hom_{\cd\ca}(A^{\wedge},\Sigma^p A^{\wedge}).\]

Let $\ca$ and $\cb$ be two dg $k$-categories. A dg
$\ca$-$\cb$-bimodule is by definition a dg module over the tensor
product $\cb \ten\ca^{op}$. Given a dg $\ca$-$\cb$-bimodule $X$, one
can define a pair of adjoint standard triangle functors
\[\xymatrix{\cd\ca\ar@<.7ex>[r]^{T_X}&\cd\cb\ar@<.7ex>[l]^{H_X}.}\]

\subsection{Orbit categories}\label{ss:orbit-category} We follow~\cite{Keller05}.

Let $\cc$ be a $k$-category, and $F$ be an
auto-equivalence of $\cc$. The orbit category $\cc/F$ is defined as
the category whose objects are the same as those of $\cc$ and the
morphism space $\Hom_{\cc/F}(X,Y)$ between two objects $X$ and $Y$
is
\[\Hom_{\cc/F}(X,Y)=\bigoplus_{p\in\mathbb{Z}}\Hom_{\cc}(X,F^p Y).\]

The following remarkable result is due to Keller.

\begin{theorem}[\cite{Keller05} Theorem 9.9]\label{t:keller's-theorem} Let $\ch$ be a small hereditary abelian $k$-category with
the Krull--Schmidt property where all morphism and extension spaces
are finite-dimensional. Let $F:\cd^b(\ch)\rightarrow \cd^b(\ch)$ be
a standard equivalence with a dg lift. Suppose
\begin{itemize}
\item[(1)] for each indecomposable $U$ of $\ch$, only finitely many
objects $F^iU,i\in\mathbb{Z}$, lie in $\ch$;
\item[(2)] there is an integer $N\geq 0$ such that the $F$-orbit of
each indecomposable of $\cc$ contains an object $\Sigma^nU$ for some
$0\leq n\leq N$ and some indecomposable $U$ of $\ch$.
\end{itemize}
Then the orbit category $\cc/F$ admits a natural triangle structure
such that the projection functor $\cc\rightarrow\cc/F$ is a triangle
functor.
\end{theorem}

In Keller's proof, a dg orbit category is defined and the triangle
structure of $\cc/F$ come from the (nice) dg structure of the dg
orbit category. In this case, we say that the orbit category $\cc/F$
\emph{admits a canonical triangle structure}. In particular, the
projection functor $\pi_{\cc}:\cc\rightarrow\cc/F$ is a triangle
functor given by a tensor functor.

\subsection{The algebraic triangulated category generated by a
spherical object}\label{ss:spherical}

Let $\cc$ be a triangulated $k$-category. For an integer $w$, an
object $S$ of $\cc$ is called a \emph{$w$-spherical object} if the
graded endomorphism algebra
\[
\bigoplus_{p\in\Z} \Hom_{\cc}(S,\Sigma^p S)
\]
is isomorphic to $\Lambda=k[s]/(s^2)$, where $s$ is of degree $w$.
Let $\cs_w$ be the algebraic triangulated category over $k$
generated by a $w$-spherical object $S$. The following result was
proved in~\cite{KellerYangZhou09}.

\begin{theorem}\label{t:spherical}
The category $\cs_w$ is triangle equivalent to $\per(\Lambda)$ and
to $\cd_{fd}(\Gamma)$, where $\Gamma=k[t]$ if $w\neq 1$, $\Gamma=k[[
t]]$ if $w=1$ with $\deg{t}=1-w$. Here both $\Lambda$ and $\Gamma$
are viewed as dg algebras with trivial differentials. In particular,
$\Sigma^n S$, $n\in\mathbb{Z}$, are precisely the $w$-spherical
objects in $\cs_w$.
\end{theorem}

Let $\Aut(\cs_w)$ be the group of triangle automorphisms of $\cs_w$
which admit dg lifts. The suspension functor $\Sigma$ belongs to
$\Aut(\cs_w)$ and is central. Let $a$ be a nonzero element of $k$.
We define $\varphi_a$ to be the automorphism of $\Lambda$ taking $s$
to $as$. The induced push-out functor
$\varphi_{a,*}:\cs_w\rightarrow\cs_w$ also belongs to $\Aut(\cs_w)$.
In fact, $\Aut(\cs_w)$ is generated by $\Sigma$ and $\varphi_{a,*}$,
$a\in k^\times$.

\begin{lemma}\label{l:picard-group-spherical} The group $\Aut(\cs_w)$ is isomorphic to $k^{\times}\times \mathbb{Z}$.
\end{lemma}
\begin{proof} Let $F$ be an element of $\Aut(\cs_w)$. Then $FS$ is also a $w$-spherical object, so it
follows from Theorem~\ref{t:spherical} that $FS\cong\Sigma^{n(F)} S$
for some $n(F)\in\mathbb{Z}$. Moreover, there is a nonzero element
$a(F)$ of $k$ such that $Fs=a(F)\Sigma^{n(F)}s$. Sending $F$ to
$(a(F),n(F))$ defines a group homomorphism $f:\Aut(\cs_w)\rightarrow
k^{\times}\times \Z$. The map $f$ is surjective because for any
$a\in k^\times$ and any $n\in\mathbb{Z}$ we have
$f(\varphi_{a,*}\Sigma^n)=(a,n)$. It is injective because any
element in the kernel is isomorphic to the identity functor on the
objects $\Sigma^p S$ ($p\in\mathbb{Z}$) and on all the morphism
spaces $\Hom_{\cs_w}(S,\Sigma^p S)$ ($p\in\mathbb{Z}$) and, since
$S$ generates $\cs_w$, such a functor must be isomorphic to the
identify functor.
\end{proof}

\subsection{Standard tubes}\label{ss:standard-tubes}

Let $n$ be a positive integer. Let $\ct_n$ be the standard tube of rank $n$, \ie the hereditary abelian category of finite-dimensional nilpotent representations of a cyclic quiver with $n$ vertices.

\begin{lemma}\label{l:dg-alg-derived-tube}
The derived category $\cd^b(\ct_n)$ is triangle equivalent to $\per(\Lambda')$, where $\Lambda'$ is the quotient of the path algebra of the graded cyclic quiver with each arrow in degree $1$ modulo all path of length $2$.
\end{lemma}
\begin{proof} It is easy to check that $\Lambda'$ is the graded endomorphism algebra of the simple objects of $\ct_n$. For degree reasons, there is no non-trivial $A_\infty$-structure on $\Lambda'$ such that the identities of the simple objects are strict units(\confer the proof of~\cite[Theorem 2.1]{KellerYangZhou09}). So by~\cite[Theorem 7.6.0.6]{Lefevre03} we obtain the derived result.
\end{proof}

Let us denote by $\alpha_1:1\rightarrow 2,\ldots,\alpha_n:n\rightarrow 1$ the arrows in the quiver of $\Lambda'$.
For $a\in k^{\times}$, we define $\psi_a$ as the unique automorphism of $\Lambda'$ taking $\alpha_1$ to $a\alpha_1$, and $\alpha_i$ to $\alpha_i$, $i=2,\ldots,n$. We define $c$ as the unique automorphism of $\Lambda'$ taking the vertex $i$ to ${i+1}$, $i=1,\ldots,n$. The push-out $c_*$ is exactly the Auslander--Reiten translation $\tau$.

\begin{lemma}\label{l:picard-group-derived-tube}
The group $\Aut(\cd^b(\ct_n))$ is isomorphic to $k^{\times}\times \mathbb{Z}/n\mathbb{Z}\times\mathbb{Z}$.
\end{lemma}

\begin{proof}
The proof is similar to that for
Lemma~\ref{l:picard-group-spherical}. Let $S_i$ denote the simple
object of $\ct_n$ corresponding to the vertex $i$, $i=1,\ldots,n$.
Let $F$ be an element of $\Aut(\cd^b(\ct_n)$. By the shape of the
Auslander--Reiten quiver of $\cd^b(\ct_n)$, we have $F(S_1)\cong
\Sigma^{n(F)}S_{i(F)}=\Sigma^{n(F)}\tau^{i(F)-1}S_1$ for some
$n(F)\in\mathbb{Z}$ and $i(F)=1,\ldots,n$. Moreover, there are
nonzero elements $a_1(F),\ldots,a_n(F)$ of $k$ such that
$F(\alpha_i)=a_i(F)\tau^{i(F)-1}\Sigma^{n(F)}\alpha_i$. Sending $F$
to $(\prod_{i=1}^n a_i(F),i(F)-1,n(F))$ defines a group homomorphism
$:\Aut(\cd^b(\ct_n))\rightarrow k^{\times}\times
\mathbb{Z}/n\mathbb{Z}\times\mathbb{Z}$. It is surjective because
for any $a\in k^{\times}$, any $i\in\mathbb{Z}/n\mathbb{Z}$ and any
$m\in\mathbb{Z}$ we have $f(\psi_{a,*}\tau^i\Sigma^m)=(a,i,m)$. Let
$F$ be in the kernel of $f$. Then $\prod_{i=1}^n a_i(F)=1$. So the
maps $a_1(F)^{-1}\cdots a_{i-1}(F)^{-1}:F(S_i)=S_i\rightarrow S_i$
defines a natural isomorphism on the generators $S_1,\ldots,S_n$,
and hence $F$ is isomorphic to the identify functor.
\end{proof}

\section{The triangle structure}\label{s:triangle-structure}

Let $k$ be a field and $w$ an integer. Let $\cs_w$ be the algebraic
triangulated category generated by a $w$-spherical object. In this
section, we will prove that for any triangle auto-equivalence $F$ of $\cs_w$ which is not identical on isoclasses of objects, the orbit
category $\cs_w/F$ admits a canonical triangle structure.

\subsection{The triangle structure}
Let
$\Gamma=\Gamma_d$ be the graded algebra $k[t]$ with
$\mathrm{deg}(t)=d=1-w$ when $w\neq 1$ or the ring of power series
$k[[t]]$ when $w=1$, viewed as a dg algebra with trivial
differential. Then $\cd_{fd}(\Gamma)\cong \cs_w$, see Theorem~\ref{t:spherical}. When $w=1$,
$\cd_{fd}(\Gamma)$ is exactly the bounded derived category of a
homogeneous tube and the orbit category $\cs_1/F$
is triangulated by Theorem~\ref{t:keller's-theorem}. So in the rest of this subsection we assume that
$w\neq 1$, \ie $d\neq 0$.

Let $n\in\mathbb{N}$. The following lemma on graded modules
over $\Gamma$ is well-known.

\begin{lemma}\label{l:indec-object}
\begin{itemize}
\item[(a)] Up to degree shifting, an indecomposable graded
$\Gamma$-module is isomorphic to one of the following modules:
$\Gamma/(t^p)~(p\in\mathbb{Z})$ , $\Gamma$, $k[t^{-1}]$ (the graded
dual of $\Gamma$), $M=k[t,t^{-1}]$.
\item[(b)] Let $X$ be a finite-dimensional graded $\Gamma$-module. Then
\[\Hom_{\Grmod(\Gamma)}(M,\bigoplus_{p\in\mathbb{Z}}X\langle np\rangle)=0=\Hom_{\Grmod(\Gamma)}(\bigoplus_{p\in\mathbb{Z}}X\langle np\rangle,M).\]
\end{itemize}
\end{lemma}

Recall that the automorphism group of $\cd_{fd}(\Gamma)$ consists of the functors $\varphi_{a,*}\Sigma^n$, $a\in k^{\times}$, $n\in\mathbb{Z}$. Here by abuse of notation $\varphi_a$ denotes the automorphism $t\mapsto at$ of $\Gamma$. It is useful to observe that $\varphi_{a,*}$ is isomorphic to the identity on objects.

\begin{theorem}\label{t:the-triangle-structure}
The orbit category $\cd_{fd}(\Gamma)/\varphi_{a,*}\Sigma^n$ admits a canonical
triangle structure.
\end{theorem}
\begin{proof}  Recall that $\Dif\Gamma$ denotes the dg category of dg $\Gamma$-modules. Let $\ca$ be the dg subcategory of $\Dif\Gamma$ consisting of strictly-perfect
dg $\Gamma$-modules with finite-dimensional total cohomology (a dg
$\Gamma$-module is strictly perfect if as a graded module it is the
direct sum of finite copies of shifts of $\Gamma$). Then we
have an equivalence of triangulated categories
\[H^0\ca\cong \cd_{fd}(\Gamma).
\]
The functor $\varphi_{a,*}\Sigma^n$ is a dg auto-equivalence of the
dg category $\ca$. For any $X, Y\in \ca$,
$\Hom_{H^0\ca}(X,(\varphi_{a,*}\Sigma^{n})^pY)$ vanishes for all but
finitely many $p\in \Z$. Let $\cb=\ca/\varphi_{a,*}\Sigma^n$ be the
dg orbit category of $\ca$. We have an equivalence of categories
\[H^0\ca/\Sigma^n\xrightarrow{\sim}H^0\cb.
\]
The canonical dg functor $\pi:\ca\to \cb$ yields an
$\ca$-$\cb$-bimodule
\[(B,A)\to \Hom_{\cb}(B,\pi A),
\]
which induces the standard functors (\confer
Section~\ref{ss:derived-category})
\[\cd\ca\xrightarrow{\pi_*}\cd\cb \ \text{and}\
\cd\cb\xrightarrow{\pi_{\rho}}\cd\ca.
\]
Note that we also have a natural embedding
\[i:\cd\ca\to \cd\Gamma
\]
given by the $\ca$-$\Gamma$-bimodule
\[(\Gamma, A)\to \Hom_{\Dif\Gamma}(\Gamma, A).
\] This
embedding identifies $H^0\ca$ with $\cd_{fd}(\Gamma)$.

Let $\cm=\per\cb$ be the triangulated subcategory of $\cd\cb$ generated by
the representable functors.  By abuse of notation, we denote the
representable functor  $X^{\wedge}$ still by $X$ for any $X\in \cb$.
To prove that $\cd_{fd}(\Gamma)/\varphi_{a,*}\Sigma^n$ admits a canonical
triangle structure, it suffices to show that $H^0\cb$ is extension
closed in $\cm$, \ie for any $X, Y\in H^0\cb$ and $f\in
\Hom_{H^0\cb}(X,Y)$, the third term of the following triangle in
$\cm$ is isomorphic to an object in $H^0\cb$
\[X\xrightarrow{f}Y\to E\to \Sigma X.
\]
We apply the
right adjoint $\pi_{\rho}$ of $\pi_*$ to the triangle above, we get a
triangle in $\cd\ca$
 \[\pi_{\rho}X\to \pi_{\rho}Y\to \pi_{\rho}E\to \Sigma\pi_{\rho}X.
 \]
 Applying the functor $i$, we have a triangle in $\cd\Gamma$
 \[i\pi_{\rho}X\to i\pi_{\rho}Y\to i\pi_{\rho}E\to
 \Sigma i\pi_{\rho}X.
 \]
 It suffices to show that
 \[i\pi_{\rho}E\cong \bigoplus_{p\in \Z}(\varphi_{a,*}\Sigma^{n})^p Z\cong\bigoplus_{p\in\Z}\Sigma^{np} Z, \ \text{for\ some }\
 Z\in \ca.
 \]
 Below we will consider the functor $\varphi_{a,*}\Sigma^n$ only on objects, so we will drop $\varphi_{a,*}$.
 Note that $i\pi_{\rho}X$ and $i\pi_{\rho}Y$ are direct sums of
 $\Sigma^n$-orbits of objects in $\cd_{fd}(\Gamma)$. We have
 \[\dim H^m(i\pi_{\rho}X)<\infty\ \text{and}\ \dim
 H^m(i\pi_{\rho}Y)<\infty, \ \forall m\in \Z.
 \] Let us rewrite  the triangle
 as follows
 \[\bigoplus_{p\in \Z}\Sigma^{np}X\to \bigoplus_{p\in \Z}\Sigma^{np}Y\to N\to \bigoplus_{p\in
 \Z}\Sigma^{np+1}X.
 \]
Applying the cohomological functor $H^*$ to the above triangle gives
a long exact sequence of graded $\Gamma$-modules
\[\bigoplus_{p\in \Z}H^*(X)\langle np\rangle\to \bigoplus_{p\in \Z}H^*(Y)\langle np\rangle\to H^*(N)\to \bigoplus_{p\in
 \Z}H^*(X)\langle np+1\rangle.\]
 Then it follows by Lemma~\ref{l:indec-object} that any degree shifting of $M=H^*(M)$ is not a
direct summand of $H^*(N)$. By Lemma~\ref{l:H^*}, we know that any
shift of $M$ is not a direct summand of $N$. As a consequence, $N$
must be the direct sum of a $\Sigma^n$-orbit. Since $\dim
H^m(N)<\infty$ for each $m\in \Z$, $N$ does only have objects in
$\cd_{fd}(\Gamma)$ as its direct summands. Again, by $\dim
H^m(N)<\infty$ for each $m\in \Z$, we have
\[N\cong (\bigoplus_{p\in\Z}\Sigma^{np}Z_1)\oplus \ldots \oplus (\bigoplus_{p\in\Z}\Sigma^{np}Z_r)
\]
for some indecomposable objects $Z_1, \ldots, Z_r$ in
$\cd_{fd}(\Gamma)$. Namely, $N$ is the direct sum of a $\varphi_{a,*}\Sigma^n$-orbit of objects in $\cd_{fd}(\Gamma)$. This finishes the proof.
\end{proof}

\begin{remark}
Recall that $\Lambda=k[s]/s^2$ is the graded algebra with
$\deg(s)=w$. In~\cite[Section 3]{Keller05} it is shown that
$\cd^b(\Lambda)/\Sigma^2\cong\per(\Gamma)/\Sigma^2$ is not
triangulated for $w=0$ (\ie $d=1$). Similarly,
$\per(\Gamma)/\Sigma^n$ is not triangulated for all $w\in\mathbb{Z}$
and all $n\in\mathbb{N}$. Indeed, we can use the argument
in~\cite[Section 3]{Keller05}. The endomorphism algebra of $\Gamma$
in the orbit category is a polynomial ring $k[u]$ with
$u\in\Hom_{\per(\Gamma)}(\Gamma,\Sigma^{l}\Gamma)$, where $|l|$ is
the least common multiple of $n$ and $|d|$ and $l$ has the same sign
as $d$. The endomorphism $1+u$ is monomorphic but does not admit a
left inverse. By copying the proof of
Theorem~\ref{t:the-triangle-structure}, we can also obtain some
evidence (and some clue about the missing cone of $1+u$). The
morphism $1+u$ induces a triangle in $\cd(\Gamma)$
\[\bigoplus_{p\in\mathbb{Z}}\Sigma^{np}\Gamma\stackrel{f}{\rightarrow}\bigoplus_{p\in\mathbb{Z}}\Sigma^{np}\Gamma\rightarrow N\rightarrow\Sigma\bigoplus_{p\in\mathbb{Z}}\Sigma^{np}\Gamma,\]
where $f$ is the morphism with components
\[\Sigma^{np}\Gamma\stackrel{(1,u)}{\longrightarrow}\Sigma^{np}\Gamma\oplus\Sigma^{np+l}\Gamma\hookrightarrow\bigoplus_{p\in\mathbb{Z}}\Sigma^{np}\Gamma.\]
Mamely, $N$ is the Milnor colimit of the sequence
\[\ldots \longrightarrow \Sigma^p\bigoplus_{i=0}^{m-1}\Sigma^i\Gamma\stackrel{\Sigma^p v}{\longrightarrow}\Sigma^{p+1}\bigoplus_{i=0}^{m-1}\Sigma^i\Gamma\longrightarrow\ldots\]
where $m$ is the greatest common divisor of $n$ and $|d|$, and $v$
is the diagonal matrix $$v=\mathrm{diag}(u,\Sigma
u,\ldots,\Sigma^{m-1}u).$$ Thus $N$ is isomorphic in $\cd(\Gamma)$
to $\bigoplus_{i=0}^{m-1}\Sigma^i k[t,t^{-1}]$, which is not the
direct sum of the $\Sigma^n$-orbit of any object in $\per(\Gamma)$.
\end{remark}

\subsection{The dg algebra for $\cs_w/\varphi_{a,*}\Sigma^n$}
We have a nice byproduct of the proof of Theorem~\ref{t:the-triangle-structure}.

\begin{proposition}\label{p:dg-algebra-for-orbit} Let $w\in\mathbb{Z}\backslash\{1\}$, $a\in k^{\times}$ and $n\in\mathbb{N}$.
 The orbit category $\cs_w/\varphi_{a,*}\Sigma^n$
 is triangle equivalent to $\per\tilde{\Lambda}$,
 where \[\tilde{\Lambda}=\tilde{\Lambda}_{w,a,n}=k\langle s,r,r^{-1}\rangle/(s^2,sr=(-1)^{nw}ars)\]
 is the graded algebra with $\deg(s)=w$ and $\deg(r)=n$, viewed as a dg algebra with trivial differential.
\end{proposition}
\begin{proof}
Let $\ca$ and $\cb$ be dg categories as defined in the proof of
Theorem~\ref{t:the-triangle-structure}. Let $S=\Gamma/t\Gamma$ be
the 1-dimensional simple dg $\Gamma$-module concentrated in degree
$0$.
 Recall from Theorem~\ref{t:spherical} that the dg endomorphism algebra
 of (a strictly perfect resolution of) $S$ in $\ca$ is related by a
 zigzag of quasi-isomorphisms to $\Lambda=k[s]/s^2$ with $\deg(s)=w$.
 Thus the dg endomorphism algebra
\[\bigoplus_{p\in\mathbb{Z}}\Hom_{\ca}(\tilde{S},(\varphi_{a,*}\Sigma^{n})^p\tilde{S})
=\bigoplus_{p\in\mathbb{Z}}\Sigma^{pn}\Hom_{\ca}(\tilde{S},\tilde{S})\]
of the image $\tilde{S}$ of $S$ in $\cb$ is related by a zigzag of
quasi-isomorphisms to $\tilde{\Lambda}$ (note that the composition
in the orbit category is twisted by $\varphi_{a,*}\Sigma^n$). It
follows from the construction that every object in $\cb$ is an
iterated cone of closed morphisms of degree 0 between shifts of
copies of $\tilde{S}$, since every object in $\ca$ is an iterated
cone of closed morphisms of degree 0 between shifts of copies of
$S$. Thus the restriction from $\cb$ to the one-object dg
subcategory $\{\tilde{S}\}$ induces a triangle equivalence
$\per\tilde{\Lambda}\cong\per\cb$, \confer for example~\cite{Guo10}.
The latter category, by the proof of
Theorem~\ref{t:the-triangle-structure}, is triangle equivalent to
$H^0\cb=\cd_{fd}(\Gamma)/\varphi_{a,*}\Sigma^n$, and the  desired
result follows.
\end{proof}

\subsection{The Auslander--Reiten quiver}

Let $a\in k^{\times}$, $n\in\mathbb{N}$, $w\in\mathbb{Z}$ and $d=1-w$.

If $w=1$, then the Auslander--Reiten quiver of $\cs_w$ consists of $\Z$ copies of homogenous tubes, and $\Sigma$ acts transitively on them. Thus the Auslander--Reiten quiver of the orbit category $\cs_1/\varphi_{a,*}\Sigma^n$ consists of $n$ homogeneous tubes.

If $w\neq 1$, then the the Auslander--Reiten quiver of $\cs_w$ consists of $|d|$ copies of $\mathbb{Z}A_{\infty}$. An object $M$ and $\Sigma^p M$ are in the same component if and only if $p$ is a multiple of $d$.  Thus the Auslander--Reiten quiver of $\cs_w/\varphi_{a,*}\Sigma^n$ consists of $m$ copies of tubes of rank $n'$, where $m$ is the greatest common divisor of $n$ and $d$ and $n'=\frac{n}{m}$.

\subsection{Orbit categories of the bounded derived category of a standard tube}

Let $n'\in\mathbb{N}$, $c\in\mathbb{Z}/n'\mathbb{Z}$, $b\in k^{\times}$ and $m\in\mathbb{N}$.

\begin{proposition}\label{p:dg-alg-orbit-derived-tube}
The orbit category $\cd^b(\ct_{n'})/\psi_{b,*}\tau^c\Sigma^m$ is triangle equivalent to $\per\tilde\Lambda'$, where
\[\tilde\Lambda'=\tilde\Lambda'_{n',b,c,m}=k\langle s,r,r^{-1}\rangle/(s^2,sr=(-1)^{n'm}b^{n'}rs)\]
is the graded algebra with $\deg(s)=w$ and $\deg(r)=n$, viewed as a dg algebra with trivial differential. Here $n=n'm$, and $w=1-md'$ for $d'$ with $cd'\equiv 1\hspace{-5pt}\pmod{n'}$.
\end{proposition}
\begin{proof} The proof is similar to that for Proposition~\ref{p:dg-algebra-for-orbit}. Let $\ca'$ be the dg category of strictly perfect dg $\Lambda'$-modules, where $\Lambda'$ was defined in Section~\ref{ss:standard-tubes}, and let $\cb'$ be the dg orbit category with respect to the dg automorphism $\psi_{b,*}\tau^c\Sigma^m$. Then $H^0\ca'=\cd^b(\ct_{n'})$ is triangulated and, by~\ref{t:keller's-theorem}, $H^0\cb'=\cd^b(\ct_{n'})/\psi_{b,*}\tau^c\Sigma^m$ is also triangulated.
The dg endomorphism algebra
\[\bigoplus_{p\in\mathbb{Z}}\Hom_{\ca'}(\tilde{S}_1,(\psi_{b,*}\tau^c\Sigma^m)^p\tilde{S}_1)
=\bigoplus_{p\in\mathbb{Z}}\Sigma^{pn}\Hom_{\ca}(\tilde{S}_1,\tilde{S}_{1+c})\]
of the image $\tilde{S}_1$ of $S_1$ in $\cb'$ is  $\tilde{\Lambda}'$. The triangulated orbit category $H^0\cb'=\cd^b(\ct_{n'})/\psi_{b,*}\tau^c\Sigma^m$ by $\tilde{S}_1$. Thus the restriction from $\cb'$ to the one-object dg subcategory $\tilde{S}_1$ induces a triangle equivalence $\per\tilde{\Lambda}'\cong\per\cb'=H^0\cb'=\cd^b(\ct_{n'})/\psi_{b,*}\tau^c\Sigma^m$.
\end{proof}

\subsection{An equivalence}

Let $n\in\mathbb{N}$, $w\in\mathbb{Z}$ and $d=1-w$. Let $m$ be the greatest common divisor of $n$ and $d$, and let $d'=\frac{d}{m}$ and $n'=\frac{n}{m}$. Let $c$ be an inverse of $d'$ modulo $n'$. Let $a,b\in k^{\times}$.

Combining Proposition~\ref{p:dg-algebra-for-orbit} and Proposition~\ref{p:dg-alg-orbit-derived-tube}, we obtain

\begin{theorem}\label{t:the-equivalence}
The two orbit categories $\cs_w/\varphi_{a,*}\Sigma^n$ and
$\cd^b(\ct_{n'})/\psi_{b,*}\tau^c\Sigma^m$ are triangle equivalent
if and only if $a=((-1)^db)^{n'}$.
\end{theorem}
\begin{proof} Notice first that
Proposition~\ref{p:dg-algebra-for-orbit} is also valid for the case
$w=1$, thanks to Proposition~\ref{p:dg-alg-orbit-derived-tube}.

The `if' part: If  $a=((-1)^db)^{n'}$, then $\tilde\Lambda'$ and
$\tilde\Lambda$ are the same dg algebra, in particular,
$\per(\tilde\Lambda')=\per(\tilde\Lambda)$, implying that
$\cs_w/\varphi_{a,*}\Sigma^n$ and
$\cd^b(\ct_{n'})/\psi_{b,*}\tau^c\Sigma^m$ are triangle equivalent.

The `only if' part: Let $F:\cd^b(\ct_{n'})/\psi_{b,*}\tau^c\Sigma^m\rightarrow\cs_w/\varphi_{a,*}\Sigma^n$ be a triangle equivalence. Then due to the shape of Auslander--Reiten quiver, $F(\tilde{S}_1)\cong\Sigma^p\tilde{S}$ for some integer $p$. Therefore the graded endomorphism algebra $\tilde{\Lambda}'$ of $\tilde{S}_1$ and $\tilde{\Lambda}$ of $\tilde{S}$ are isomorphic, which implies that $a=((-1)^db)^{n'}$.
\end{proof}

In the appendix, we will construct an explicit equivalence for the case $a=b=1$
and $n=2$ using covering and the universal property of orbit
categories.

\begin{example} Let $w=2$ and $n\in\mathbb{N}$. Then
$\cs_2=\cc_{\overrightarrow{A}^{\infty}_{\infty}}$ is known as the
cluster category of
$\overrightarrow{A}^{\infty}_{\infty}$~\cite{KellerReiten08}~\cite{HolmJoergensen09}.
By Theorem~\ref{t:the-equivalence}, when $n$ is even, the orbit
category $\cc_{\overrightarrow{A}^{\infty}_{\infty}}/\Sigma^n$ and
the cluster tube $\cc_{n}=\cd^b(\ct_n)/\tau^{-1}\circ\Sigma$ of rank
$n$ are triangle equivalent, while when $n$ is odd, they are not
triangle equivalent.
\end{example}

\section{Ringel--Hall Lie algebras associated to  spherical
objects}\label{s:Ringel-Hall}

Let $k$ be a finite field with $|k|=q$ and $w$ be an integer. Let
$\cs_w$ be the triangulated category over $k$ generated by a
$w$-spherical object. As shown by
Theorem~\ref{t:the-triangle-structure}, the orbit category
$\cs_w/\Sigma^2$ admits a canonical triangle structure. It is
obviously 2-periodic, so we can associate a Lie algebra to it via the
Ringel--Hall approach in the sense of Peng--Xiao. In this section,
we will determine this Lie algebra.

\subsection{The Ringel--Hall Lie algebra}
We recall the definition of the Ringel--Hall Lie algebra of a
$2$-periodic triangulated category following~\cite{PengXiao2000}.
Let $\mathcal{R}$ be a Hom-finite $k$-linear triangulated category
with suspension functor $\Sigma$. By $\ind \mathcal{R}$ we denote a
set of representatives of the isoclasses of all indecomposable
objects in $\mathcal{R}$.

Given any objects $X,Y,L$ in $\mathcal{R}$, we define
\begin{eqnarray*}W(X,Y;L)&=&\{(f,g,h)\in \Hom_{\mathcal{R}}(X,L)\times \Hom_{\mathcal{R}}(L,Y)\times \Hom_{\mathcal{R}}(Y,\Sigma X)|\\
&&X\xrightarrow{f}L\xrightarrow{g} Y\xrightarrow{h}\Sigma X \text{
is a triangle}\}
\end{eqnarray*}
The action of $\Aut(X)\times \Aut(Y)$ on $W(X,Y;L)$ induces the orbit space
\[V(X,Y;L)=\{(f,g,h)^{\wedge}|(f,g,h)\in W(X,Y;L)\}
\]
where
\[(f,g,h)^{\wedge}=\{(af, gc^{-1}, ch(\Sigma a)^{-1})|(a,c)\in \Aut(X)\times\Aut(Y)\}.
\]
Let $\Hom_{\mathcal{R}}(X,L)_Y$ be the subset of $\Hom_{\mathcal{R}}(X,L)$ consisting of
 morphisms $l:X\to L$ whose mapping cone $Cone(l)$ is isomorphic to $Y$. Consider the action of the
 group $\Aut(X)$ on $\Hom_{\mathcal{R}}(X,L)_Y$ by $d\cdot l=dl$, the orbit is denoted by $l^*$ and the
 orbit space is denoted by $\Hom_{\mathcal{R}}(X,L)_Y^*$.
  Dually one can also consider the subset $\Hom_{\mathcal{R}}(L,Y)_{\Sigma X}$
  of $\Hom_{\mathcal{R}}(L,Y)$ with the group action $\Aut(Y)$ and the orbit
   space $\Hom_{\mathcal{R}}(L,Y)_{\Sigma X}^*$. The following proposition is an observation due to
   ~\cite{XiaoXu2008}.
\begin{proposition}\label{p:hall-num}
$|V(X,Y;L)|=|\Hom_{\mathcal{R}}(X,L)_Y^*|=|\Hom_{\mathcal{R}}(L,Y)_{\Sigma X}^*|$.
\end{proposition}

We assume further that $\mathcal{R}$ is \emph{$2$-periodic}, \ie
$\mathcal{R}$ is Krull--Schmidt and $\Sigma^2\cong 1$.

Let $\Gr(\mathcal{R})$ be the Grothendieck group of $\mathcal{R}$
and $I_{\mathcal{R}}(-,-)$ be the symmetric Euler form of
$\mathcal{R}$. For an object $M$ of $\mathcal{R}$, we denote by
$[M]$ the isoclass of $M$ and by $h_M=\dimv M$ the canonical image
of $[M]$ in $\Gr(\mathcal{R})$. Let $\mh$ be the subgroup of
$\Gr(\mathcal{R})\otimes_{\Z}\Q$ generated by $\frac{h_M}{d(M)},
M\in \ind \mathcal{R}$, where $d(M)=\dim_k(\End(X)/rad \End(X))$.
One can naturally extend the symmetric Euler form to $\mh\times
\mh$. Let $\mn$ be the free abelian group with basis $\{u_X|X\in
\ind \mathcal{R}\}$. Let
\[\mg(\mathcal{R})=\mh\oplus\mn,
\]
a direct sum of $\Z$-modules. Consider the quotient group
\[\mg(\mathcal{R})_{(q-1)}=\mg(\mathcal{R})/(q-1)\mg(\mathcal{R}).
\]
Let $F_{YX}^L=|V(X,Y;L)|$. Then by Peng and Xiao
~\cite{PengXiao2000} we know that $\mg(\mathcal{R})_{(q-1)}$ is a
Lie algebra over $\Z/(q-1)\Z$, called the \emph{Ringel--Hall Lie
algebra} of $\mathcal{R}$. The Lie operation is defined as follows.
\begin{itemize}
\item[(1)] for any indecomposable objects $X,Y\in \mathcal{R}$,
\[
[u_X,u_Y]=\sum_{L\in \ind
\mathcal{R}}(F_{YX}^L-F_{XY}^L)u_L-\delta_{X,\Sigma
Y}\frac{h_X}{d(X)},
\]
where $\delta_{X,\Sigma Y}=1$ for $X\cong \Sigma Y$ and $0$ else.
\item[(2)] $[\mh, \mh]=0.$
\item[(3)] for any objects $X,Y\in \mathcal{R}$ with $Y$ indecomposable,
\[[h_X,u_Y]=I_{\mathcal{R}}(h_X,h_Y)u_Y,\qquad [u_Y,
h_X]=-[h_X,u_Y].
\]
\end{itemize}

If the triangulated category $R$ is \emph{proper}, \ie any
indecomposable object of $R$ has non-trivial class in the
Grothendieck group, then the sum over $L\in\ind R$ in (1) is
necessarily zero.

\subsection{The Ringel--Hall Lie algebra of $\cs_w/\Sigma^2$: the case $w$ is odd}

Applying Theorem~\ref{t:the-equivalence} to the case $a=1$, $w$ is odd and
$n=2$, we have the following lemma.

\begin{lemma}\label{l:periodic-characterization-w-odd}
The $2$-periodic orbit categories $\cs_w/\Sigma^2$ is triangulated
equivalent to the root category $\cd^b(\ct_1)/\Sigma^2$ of the
standard homogenous tube $\ct_1$.
\end{lemma}

We recall that the standard homogeneous tube $\ct_1$ is the category
of finite-dimensional nilpotent representations of the Jordan quiver
(the quiver with one vertex and one loop). For each positive integer
$n$ there is an indecomposable representation $\langle n\rangle$ of
length $n$, and up to isomorphism all indecomposable representations
are of this form. Let $\mathcal{R}=\cd^b(\ct_1)/\Sigma^2$ denote the
root category of $\ct_1$. The image of $\langle n\rangle$ in
$\mathcal{R}$ will still be denoted by $\langle n\rangle$, and its
suspension will be denoted by $\langle -n\rangle$. The Grothendieck
group of $\mathcal{R}$ is free of rank 1 generated by the canonical
image $z$ of $\langle 1\rangle$. To compute the Ringel--Hall Lie
algebra of $\mathcal{R}$, we need the following well-known fact on
the bounded derived category of a hereditary abelian category.

\begin{lemma}\label{l:decomposable-morphism}
Let $\ca$ be a hereditary abelian category and $\cd^b(\ca)$ the bounded derived category. Let $f:X\to Y$ be a morphism of $\ca$, then
\[X\xrightarrow{f}Y\xrightarrow{(\pi,u)}\cok f\oplus \Sigma \ker f \xrightarrow{(v,l)'} \Sigma X\]
is a triangle of $\cd^b(\ca)$.
\end{lemma}

The following characterization of the Ringel--Hall Lie algebra of
$\mathcal{R}$ is known to the experts, but we failed to find a
precise statement in the literature. It can be easily deduced by
using Lemma~\ref{l:decomposable-morphism} (\confer
~\cite{Pengxiao1998}~\cite{Zhang1997}).

\begin{proposition}~\label{p:w-odd}
The Ringel--Hall Lie algebra $\mg(\mathcal{R})_{q-1}$ of
$\mathcal{R}$ has a $\mathbb{Z}/(q-1)$-basis $\{u_{\langle
n\rangle}|n\in\mathbb{Z}\backslash\{0\}\}\cup\{z\}$ with structure
constants given by
\begin{itemize}
\item[(1)] $[u_{\langle n\rangle}, u_{\langle-n\rangle}]=-nz$, for $n\in \Z\backslash\{0\}$;
\item[(2)] $[u_{\langle n\rangle}, u_{\langle m\rangle}]=0, $ for $n\neq m \in \Z\backslash\{0\}$;
\item[(3)] $[z, u_{\langle n\rangle}]=0$, for $n\in \Z\backslash\{0\}$.
\end{itemize}
\end{proposition}

We have an `integral' version of $\mg(\mathcal{R})_{q-1}$. Let $\Omega$ be the set of isomorphism classes of finite field
extensions of $k$. For any $E\in \Omega$,  let $\langle
n\rangle^E=\langle n\rangle\otimes_k E$ be the extension of $\langle
n\rangle$ over field $E$ and $\langle-n\rangle^E$ be the suspension
of $\langle n\rangle^E$. One can define the Ringel--Hall Lie algebra
$\mg(\mathcal{R})_{(|E|-1)}$ similarly. Consider the product of Lie
algebras
\[\mathcal{L}:=\prod_{E\in \Omega}\mg(\mathcal{R})_{|E|-1}.
\]
Let $U_{\langle\pm n\rangle}=(\ldots, u_{\langle\pm
n\rangle^{E}},\ldots)_E$ ($n\in\mathbb{N}$), $Z=(\ldots,
z^E,\ldots)_E$. Consider the Lie subalgebra $\mg$ of $\mathcal{L}$
generated by $U_{\langle\pm n\rangle}$ and $Z$, we  also call $\mg$
the Ringel--Hall Lie algebra of $\mathcal{R}$. It is easy to see
that $\mg$ is isomorphic to the infinite-dimensional Heisenberg Lie
algebra. Indeed, $\{-\frac{1}{n}U_{\langle
n\rangle}|n\in\mathbb{N}\}\cup\{Z\}\cup\{U_{\langle-n\rangle}|n\in\mathbb{N}\}$
is a Chevalley basis.

Combining Lemma~\ref{l:periodic-characterization-w-odd} and
Proposition~\ref{p:w-odd}, we have
\begin{proposition}
The Ringel--Hall algebra of $\cs_w/\Sigma^2$ for $w$ odd is
isomorphic to the infinite-dimensional Heisenberg Lie algebra.
\end{proposition}

\subsection{The Ringel--Hall Lie algebra of $\cs_w/\Sigma^2$: the case $w$ is even}

Applying Theorem~\ref{t:the-equivalence} to the case $a=1$, $w$ is even
and $n=2$, we have the following lemma.

\begin{lemma}\label{l:periodic-characterization-w-even}
The $2$-periodic orbit categories $\cs_w/\Sigma^2$ is triangulated
equivalent to the cluster tube of rank $2$.
\end{lemma}

Let us first recall the definition of the cluster tube of rank $2$.
Let $\Delta$ be the cyclic quiver with $2$ vertices. Let $\ct_2$ be
the category of finitely generated nilpotent right
$k\Delta$-modules. Let $\cd=\cd^b(\ct_2)$ be the bounded derived
category of $\ct_2$, $\tau$ be the AR-translation functor, and
$\Sigma$ be the suspension functor of $\cd^b(\ct_2)$. The
\emph{cluster tube of rank $2$}, denoted by $\cc$, is defined as the
orbit category $\cd^b(\ct_2)/\tau^{-1}\circ\Sigma$ (\confer
\cite{BarotKussinLenzing08}\cite{BuanMarshVatne10}). In particular,
for objects $X$ and $Y$ of $\cc$, the morphism space $\cc(X,Y)$ is
\[\cc(X,Y)=\cd(X,Y)\oplus \cd(X,\tau^{-1}\Sigma Y).
\]
The composite functor $\ct_2\rightarrow\cd\rightarrow\cc$ is
bijective on isoclasses of objects and preserves indecomposability.
Thus we have $\ind \cc=\ind \ct_2=\{\langle n\rangle,\langle
-n\rangle|n\in \N\}$, where $\langle n\rangle$ is the unique
indecomposable $k\Delta$-module of length $n$ with socle the simple
module corresponding to the vertex $1$, and $\langle -n\rangle$ the
unique indecomposable $k\Delta$-module of length $n$ with socle the
simple module corresponding to the vertex $2$. We have $\tau \langle
n\rangle=\langle -n\rangle$ in $\ct_2$ (here $\tau^2=1$), and
$\Sigma\langle n\rangle\cong\tau\langle n\rangle =\langle -n\rangle$
in $\cc$ (here $\tau\cong\Sigma$ and $\tau^2\cong \Sigma^2=1$). The
Grothendieck group of $\cc$ is free of rank 1 generated by the
canonical image of $\langle 1\rangle$,
see~\cite{BarotKussinLenzing08}. It is easy to see that the images
of $\langle \pm 2n\rangle$ in the Grothendieck group are 0, so the
triangulated category $\cc$ is not proper.

To compute the Ringel--Hall Lie algebra of $\cc$, we need some
auxiliary results. Let $X$, $Y$ be two objects of $\ct_2$, viewed as
objects of $\cc$, and $f\in \cd(X,Y)$ and $g\in \cd(X,\tau^{-1}Y)$.
Let $Z$ be an object of $\ct_2$ such that $\Sigma Z$ is a cone of
$f+g$. As discussed in~\cite[Section 4.6]{Yang10a}, we can compute
$Z$ using the long exact sequence
\[\xymatrix{
\tau^{-1}X\ar[r]^{\tau^{-1}f} & \tau^{-1}Y \ar[r]& Z \ar[r] &
X\ar[r]^{f} & Y ,}\] where the short exact sequence
\[\xymatrix{0\ar[r] & \cok(\tau^{-1}f)\ar[r] & Z \ar[r] & \ker(f)\ar[r] &
0}\] is induced from $g$ by the inclusion $\ker(f)\hookrightarrow X$
and the quotient $\tau^{-1}Y\twoheadrightarrow \cok(\tau^{-1}f)$.
Namely, we have the following commutative diagram
\[\xymatrix{
g: & 0\ar[r] & \tau^{-1}Y\ar[r]\ar@{->>}[d] & E\ar[r]\ar[d] & X\ar[r]\ar@{=}[d] & 0\\
& 0\ar[r] & \cok(\tau^{-1}f)\ar[r] & E'\ar[r] & X\ar[r] & 0\\
& 0\ar[r] & \cok(\tau^{-1}f)\ar[r]\ar@{=}[u] & Z\ar[u] \ar[r] &
\ker(f)\ar[r]\ar@{^{(}->}[u] & 0~~,}\] where the square in the
left-upper corner is a pushout and the square in the right-lower
corner is a pullback. The next proposition follows easily.

\begin{proposition}\label{p:indecomposable}
Let $f\in \cd(X,Y), g\in \cd(X,\tau^{-1}\Sigma Y)$. For any nonzero
$t\in k$, we have $Con(f+tg)\cong Con(f+g)$. In particular,
$Con(f+g)$ is indecomposable if and only if $Con(f+tg)$ is
indecomposable.
\end{proposition}

For $f\in \cd(X,Y)$ and $g\in \cd(X,\tau^{-1}\Sigma Y)$, let
$(f+g)^*$ be the orbit of $f+g$ by the action of $\Aut_{\cc}(X)$ and
$(f+g)_*$ be the orbit of $f+g$ by the action of $\Aut_{\cc}(Y)$. We
have the following easy observations
\begin{itemize}
\item[$\cdot$] for any $g\in \cd(X, \tau^{-1}\Sigma Y)$ and any $h\in \cd(Y, \tau^{-1}\Sigma
Z)$, we have $h\circ g=0$ in $\cc$;
\item[$\cdot$] for $g\in\cd(X,\tau^{-1}\Sigma Y)$, the orbit $g^*$ is
contained in $\cd(X,\tau^{-1}\Sigma Y)$.
\end{itemize}

\begin{lemma}\label{l:inj-surj}
Let $X,Y\in\ind \cc$, $f\in \cd(X,Y)$, $g\in \cd(X, \tau^{-1}\Sigma
Y)$. Assume that $(f+g)^*=h^*$ for some $h\in \cd(X,Y)$, then the
mapping cone $Con(f+g)$ of $f+g$ in $\cc$ is indecomposable if and
only if $f$ is injective or surjective.
\end{lemma}
\begin{proof}
Suppose $Con(f+g)$ is indecomposable, then by the condition
$(f+g)^*=h^*$, we know that $Con(h)$ is indecomposable. Thus, $h$ is
injective or surjective. Since $(f+g)^*=h^*$, there exists
$\phi_X\in \Aut_{\cc}(X)$ such that $(f+g)\phi_X=h$. We can write
$\phi_X$ as $\phi+x$, where $\phi\in \Aut_{\cd}(X)$ and $x\in \cd(X,
\tau^{-1}\Sigma X)$. In particular, we have $f\circ \phi=h$ and $f$
is injective or surjective.

Suppose $f$ is injective (resp. surjective). The assumption
$(f+g)^*=h^*$  implies that $h$ is injective (resp. surjective).
Therefore the mapping cone of $f+g$ is indecomposable.
\end{proof}
\begin{remark}
For any $f\in \cd(X,Y), g\in \cd(X, \tau^{-1}\Sigma Y)$, if $f$ is
surjective, then $(f+g)^*=f^*$ ; if $f$ is injective, then
$(f+g)_*=f_*$.
\end{remark}

\begin{lemma}\label{l:noniso}
Let $X,Y\in \ind \cc$ and  $f\in \cd(X,Y), g\in
\cd(X,\tau^{-1}\Sigma Y)$. Suppose $(f+g)^*\neq h^*$ for any $h\in
\cd(X,Y)$. Then $(f+sg)^*\neq (f+tg)^*$ for any nonzero $s\neq t\in
k$.
\end{lemma}

\begin{proof}
We first remark that if $(f+g)^*\neq h^*$ for any $h\in \cd(X,Y)$,
then $(f+tg)^*\neq h^*$ for any $h\in \cd(X,Y)$ and $0\neq t\in k$.

Suppose that $(f+sg)^*=(f+tg)^*$ for some $s\neq t\in k$. There
exists $\phi_X\in \Aut_{\cc}(X)$ such that $(f+sg)\phi_X=(f+tg)$.
Since $X$ is indecomposable, we can write $\phi_X=a+b+c$, where
$a\in k, b\in rad \End_{\cd}(X), c\in \cd(X,\tau^{-1}\Sigma X)$. The
equality $(f+sg)(a+b+c)=(f+tg)$ implies $f(a+b)=f$ and
$-fc=g(s(a+b)-t)$. In particular, one has $a=1$ and $s-t+sb\in
\Aut_{\cd}(X)$ since $s\neq t$ and $sb\in rad \End_{\cd}(X)$. Thus,
we have $g=-fc(s-t+sb)^{-1}$ and $(f+g)(1-c(s-t+sb)^{-1})=f$ which
is a  contradiction.
\end{proof}

\begin{lemma}\label{l:iso}
Let $X,Y\in \ind \cc$ and  $f,h\in \cd(X,Y), g,k\in
\cd(X,\tau^{-1}\Sigma Y)$. If $(f+g)^*=(h+k)^*$, then
$(f+tg)^*=(h+tk)^*$ for any $0\neq t\in k$.
\end{lemma}
\begin{proof}
Assume that $\phi_X=a+b+c$, where $a\in k, b\in rad \End_{\cd}(X),
b\in \cd(X,\tau^{-1}\Sigma X)$ and $(f+g)\phi_X=h+k$. An easily
calculation shows that $f(a+b)=h$ and $fc+g(a+b)=k$. One can take
$\psi=a+b+tc$ and verify that $(f+tg)\psi=(h+tk)$.
\end{proof}

Let $X$, $Y$ and $L$ be objects in $\ind\cc$. By definition and
Proposition~\ref{p:hall-num}, we have
$F_{YX}^L=|\Hom_{\cc}(X,L)^*_Y|$. The set $\Hom_{\cc}(X,L)^*_Y$
admits a natural partition:
\[\Hom_{\cc}(X,L)^*_Y=S_1\cup S_2\cup S_3,\]
where \begin{eqnarray*} S_1&=&\{f^*|f\in\cd(X,L) ~~s.t.~~
Con(f)\cong
Y\},\\
S_2&=&\{(f+g)^*|0\neq f\in\cd(X,L),g\in\cd(X,\tau^{-1}\Sigma L)
~~s.t.~~
Con(f+g)\cong Y\}\backslash S_1,\\
S_3&=&\{g^*|0\neq g\in\cd(X,\tau^{-1}\Sigma L)~~s.t.~~Con(g)\cong
Y\}.
\end{eqnarray*}
Therefore
\[F_{YX}^L=|S_1|+|S_2|+|S_3|.\]
It follows from Proposition~\ref{p:indecomposable},
Lemma~\ref{l:noniso} and Lemma~\ref{l:iso} that $|S_2|$ is divisible
by $q-1$ and hence is $0$ in $\mathbb{Z}/(q-1)$. However, note that
$S_2$ is not necessarily empty. Namely, for some $X$ and $L$ there
are nonzero morphisms $f\in \cd(X,L)$ and $g\in
\cd(X,\tau^{-1}\Sigma L)$ such that $(f+g)^*\neq h^*$ for any $h\in
\cd(X,L)$  and the mapping cone of $f+g$ in $\cc$ is indecomposable.
For example, one considers $X=\langle 4\rangle$ and $L=\langle
3\rangle$. Let $f:X\twoheadrightarrow \langle 2\rangle
\rightarrowtail L$ and $g\in \cd(X,\tau^{-1}\Sigma L)$ given by the
short exact sequence
\[\xymatrix{0\ar[r] &\langle -3\rangle \ar[r] &\langle -7\rangle\ar[r] & \langle 4\rangle \ar[r] & 0.}\]
One checks that the cone of $f+g$ is $\langle
-3\rangle$.

It is not hard to see that
\begin{eqnarray*}
 S_1&=&\Hom_{\cd}(X,L)^*_Y~\cup~\Hom_{\cd}(X,L)^*_{\Sigma\tau^{-1}Y},\\
S_3&=&\Hom_{\cd}(X,\Sigma\tau^{-1}L)^*_{\Sigma\tau^{-1}Y}.
\end{eqnarray*} Thus the numbers $|S_1|$ and $|S_3|$ are essentially characterized by the
following proposition due to Peng--Xiao \cite{Pengxiao1998}.
\begin{proposition}\label{p:v}
Let $X,Y,Z$ be indecomposable modules of $\ct_2$. Then in the
triangulated category $\cd=\cd^b(\ct_2)$, we have
\begin{itemize}
\item[1)] $|V_{\cd}(X,Y;Z)|=1$ if and only if there is a triangle of the form
$X\to Z\to Y\to \Sigma X$; otherwise $|V_{\cd}(X,Y;Z)|=0$;
\item[2)] $|V_{\cd}(Z,\Sigma X;Y)|=\frac{|V_{\cd}(X,Y;Z)||\Aut_{\cd}
(Y)||\Hom_{\cd}(Z,X)|^2}{|\Aut_{\cd}(Z)||\Hom_{\cd}(Y,X)|}$;
\item[3)] $|V_{\cd}(\Sigma^{-1}
Y,Z;X)|=\frac{|V_{\cd}(X,Y;Z)||\Aut_{\cd}(X)||\Hom_{\cd}(Y,Z)|^2}{|\Aut_{\cd}(Z)||\Hom_{\cd}(Y,X)|}$.
\end{itemize}
\end{proposition}

Now we are in a position to determine the structure of the Ringel--Hall Lie algebra
$\mg(\cc)_{(q-1)}$ of $\cc\cong \cs_w/\Sigma^2$ for $w$ even. Let
$\mathfrak{L}$ be the Lie algebra over $\mathbb{Z}/(q-1)$ with basis
$\{z, u_n,n\in \Z\backslash\{0\}\}$ and structure constants
given by
\begin{itemize}
\item[1)] $[u_m,u_n]=0$ for $m$ and $n$ even;
\item[2)] $[u_m,u_n]=0$ for $m$ and $n$ both odd of the same sign;
\item[3)] $
[u_{2x},u_{2y-1}]=\left
\{\begin{array}{ll}u_{2(x+y)-1}+u_{2(y-x)-1},& x<
y\\u_{2(x+y)-1}-u_{2(x-y)+1},&x\geq y; \end{array}\right.$
\item[4)] $[u_{2x},u_{-2y+1}]=\left\{\begin{array}{ll}-u_{-2(x+y)+1}-u_{2(x-y)+1}, &x< y\\-u_{-2(x+y)+1}+u_{2(y-x)-1}, &x\geq y;\end{array}\right.$
\item[5)] $[u_{-2x},u_{2y-1}]=\left\{\begin{array}{ll}-u_{2(x+y)-1}-u_{2(y-x)-1}, &x< y\\-u_{2(x+y)-1}+u_{2(x-y)+1}, &x\geq y;\end{array}\right.$
\item[6)] $
[u_{-2x},u_{-2y+1}]=\left
\{\begin{array}{ll}u_{-2(x+y)+1}+u_{2(x-y)+1},& x<
y\\u_{-2(x+y)+1}-u_{2(y-x)-1},&x\geq y; \end{array}\right.$
\item[7)]$[u_{2x-1},u_{-2y+1}]=\left\{\begin{array}{ll}u_{2x+2y-2}-u_{-2x-2y+2}+u_{2x-2y}-u_{2y-2x}, & x<y\\-z+u_{4x-2}-u_{-4x+2},&x=y\\u_{2x+2y-2}-u_{-2x-2y+2}+u_{2y-2x}-u_{2x-2y},&x>y;\end{array}\right.$
\item[8)] $[z,u_{n}]=\left\{\begin{array}{ll}0,\text{ for $n$ even}\\4u_n,\text{ for $n$ positive odd}\\-4u_n,\text{ for $n$ negative odd}\end{array}\right.$
\end{itemize}
where $x,y\in \N$.

\begin{theorem}\label{t:w-even}
The assignment $z\mapsto h_{\langle 1\rangle},~ u_{n}\mapsto
u_{\langle n\rangle}, n\in\mathbb{Z}\backslash\{0\}$ linearly
extends to a Lie algebra isomorphism from $\mathfrak{L}$ to
$\mg(\cc)_{(q-1)}$.
\end{theorem}
\begin{proof}
 Let us check that $u_{\langle n\rangle},n\in\mathbb{Z}\backslash\{0\}$ satisfy 3). Similarly one checks that they together with
$h_{\langle 1\rangle}$ satisfy other relations. Let
$x,y\in\mathbb{N}$ and $l\in\mathbb{Z}\backslash\{0\}$. By the
arguments before Proposition~\ref{p:v}, we have
\begin{eqnarray*}
 F_{\langle 2y-1\rangle,\langle 2x\rangle}^{\langle l\rangle} &=& |V_{\cd}(\langle 2x\rangle,\langle 2y-1\rangle;\langle l\rangle)| + |V_{\cd}(\langle 2x\rangle,\Sigma\tau^{-1}\langle 2y-1\rangle;\langle l\rangle)|\\
&&\mbox{} + |V_{\cd}(\langle 2x\rangle,\Sigma\tau^{-1}\langle 2y-1\rangle;\Sigma\tau^{-1}\langle l\rangle)|\\
&=&|V_{\cd}(\langle 2x\rangle,\langle 2y-1\rangle;\langle l\rangle)| + |V_{\cd}(\langle 2x\rangle,\Sigma\tau^{-1}\langle 2y-1\rangle;\langle l\rangle)|\\
&&\mbox{}  + |V_{\cd}(\Sigma^{-1}\tau\langle 2x\rangle,\langle 2y-1\rangle;\langle l\rangle)|\\
&=&|V_{\cd}(\langle 2x\rangle,\langle 2y-1\rangle;\langle l\rangle)| + |V_{\cd}(\langle 2x\rangle,\Sigma\langle -2y+1\rangle;\langle l\rangle)|\\
&&\mbox{}  + |V_{\cd}(\Sigma^{-1}\langle -2x\rangle,\langle 2y-1\rangle;\langle l\rangle)|\\
&\equiv&|V_{\cd}(\langle 2x\rangle,\langle 2y-1\rangle;\langle l\rangle)| + |V_{\cd}(\langle -2y+1\rangle,\langle l\rangle;\langle 2x\rangle)|\\
&&\mbox{}  + |V_{\cd}(\langle l\rangle,\langle -2x\rangle;\langle
2y-1\rangle)| \pmod{q-1},
\end{eqnarray*}
where the last congruence follows from Proposition~\ref{p:v} 2) 3) and the fact that the number of automorphisms of an indecomposable object in $\cd$ is the product of $q-1$ and a power of $q$. Now it follows from Proposition~\ref{p:v} 1) that
\begin{eqnarray*}
 |V_{\cd}(\langle 2x\rangle,\langle 2y-1\rangle;\langle l\rangle)|&=&\begin{cases} 1 & \text{ if } l=2(x+y)-1\\0 &\text{ else};\end{cases}\\
|V_{\cd}(\langle -2y+1\rangle,\langle l\rangle;\langle 2x\rangle)|&=& 0;\\
|V_{\cd}(\langle l\rangle,\langle -2x\rangle;\langle
2y-1\rangle)|&=&\begin{cases} 1 & \text{ if } x< y \text{ and }
l=2(y-x)-1\\ 0 & \text{ else}.\end{cases}
\end{eqnarray*}
Therefore, we have
\begin{eqnarray*}
 F_{\langle 2y-1\rangle,\langle 2x\rangle}^{\langle l\rangle} &\equiv& \begin{cases} 1 & \text{ if } l=2(x+y)-1\\ 1 & \text{ if } x< y \text{ and } l=2(y-x)-1\\ 0 &\text{ else}.\end{cases}\pmod{q-1}\end{eqnarray*}
Similarly, we have
\begin{eqnarray*}
 F_{\langle 2x \rangle, \langle 2y-1 \rangle}^{\langle l \rangle} &\equiv& \begin{cases} 1 & \text{ if } x\geq y \text{ and } l=2(x-y)+1\\ 0&\text{ else}.\end{cases}\pmod{q-1}
\end{eqnarray*}
The desired result follows immediately from the definition of the bracket
\begin{eqnarray*}
 [u_{\langle 2x\rangle},u_{\langle 2y-1\rangle}]&=&\sum_{l\in\mathbb{Z}\backslash\{0\}}(F_{\langle 2y-1\rangle,\langle 2x\rangle}^{\langle l\rangle}-F_{\langle 2x \rangle, \langle 2y-1 \rangle}^{\langle l \rangle})u_{\langle l\rangle}.
\end{eqnarray*}
\end{proof}

\begin{remark} 1) 3) 4) 5) 6) 8) imply that each $u_{2x}+u_{-2x}$
($x\in\mathbb{N}$) is central in $\mathfrak{L}$. A direct
computation shows that they form a basis of the center.
\end{remark}

From Theorem~\ref{t:w-even} we see that the integral Ringel--Hall
Lie algebra of $\cc$ is isomorphic the Lie algebra over $\mathbb{Z}$
with basis $\{z,u_n,n\in\mathbb{Z}\backslash\{0\}\}$ and structure
constants given by 1)--8). We take the quotient of this Lie algebra by its
center and extend the scalars to $\mathbb{Q}$. The resulting Lie algebra has a basis $\{a_x|x\in\mathbb{N}\cup\{0\}\}\cup \{b_y,c_y|y\in\mathbb{N}-\frac{1}{2}\}$ ($a_x=\bar
u_{2x}$ for $x\in\mathbb{N}$, $a_0=\frac{1}{2}\bar{z}$, $b_y=\bar u_{2y},c_y=\bar u_{-2y}$ for $y\in\mathbb{N}-\frac{1}{2}$) with structure
constants
\begin{itemize}
\item[$\cdot$] $[a_x,a_{x'}]=0$, $[b_y,b_{y'}]=0$, $[c_y,c_{y'}]=0$;
\item[$\cdot$] $[a_x,b_y]=b_{y+x}+\mathrm{sgn}(y-x)b_{|y-x|}$,
$[a_x,c_y]=-c_{y+x}-\mathrm{sgn}(y-x)c_{|y-x|}$;
\item[$\cdot$] $[b_y,c_{y'}]=2a_{y+y'}-2a_{|y-y'|}$.
\end{itemize}
where for an integer $r$, $\mathrm{sgn}(r)=1$ if $r$ is positive and $\mathrm{sgn}(r)=-1$ if $r$ is negative.

\section{Appendix: An explicit equivalence}
Let $k$ be a field and $d$ be a
nonzero integer. Let $\Gamma$ be the graded algebra $k[t]$ with
$\mathrm{deg}(t)=d$, viewed as a dg algebra with trivial
differential. Let $n$ be a positive integer. In Section~\ref{s:triangle-structure} we proved that the orbit category of
$\cd_{fd}(\Gamma)/\Sigma^n$ admits a canonical triangle structure and is triangle equivalent to a certain orbit category of the bounded derived category of a standard tube.
In this appendix, we construct an explicit equivalence,
provided that $n$ is even.

\subsection{Induced functors}\label{ss:induced-functor} We follow~\cite{Keller08d}. Let $\cc$
and $\cc'$ be a $k$-linear category, $F:\cc\rightarrow\cc$ and
$F':\cc'\rightarrow\cc'$ be auto-equivalences, and $\cc/F$ and
$\cc'/F'$ be the corresponding orbit categories. Let $(\Phi,\alpha)$
be an \emph{$(F,F')$-equivariant functor}, \ie
$\Phi:\cc\rightarrow\cc'$ is a $k$-linear functor and $\alpha:\Phi
F\rightarrow F'\Phi$ is a natural isomorphism. Then $(\Phi,\alpha)$
induces a $k$-linear functor $\bar{\Phi}:\cc/F\rightarrow\cc'/F'$:
for $f\in \Hom_{\cc}(X,F^p Y)$ the image $\bar{\Phi}(f)$ is the
composition
\[\xymatrix@!C=1.5pc{\Phi X\ar[rr]^{\Phi f} && \Phi F^pY \ar[rr]^{\alpha_{\scriptstyle{F^{p-1}Y}}} && F\Phi F^{p-1}Y \ar[rr]^(.6){F\alpha_{\scriptstyle{F^{p-2}Y}}} && \cdots\ar[rr]^{F^{p-1}\alpha_{\scriptstyle{Y}}} && F^p\Phi Y.}\]
In particular, the $(F,F)$-equivariant functor
$(\mathrm{id}_\cc,\epsilon\id_F)$ induces a functor
$\cc/F\rightarrow\cc/F$, denoted by $\Delta(\epsilon)$. If $\cc$ is
triangulated and $F$ is a triangle auto-equivalence, then the
suspension functor of $\cc/F$ is induced by the $(F,F)$-equivariant
functor $(\Sigma,\phi)$, where $\phi:\Sigma F\rightarrow F\Sigma$ is
the natural isomorphism in the triangle structure of $F$.

Let $F_1,F_2:\cc\rightarrow\cc$ be two $k$-linear functors endowed
with commutation morphisms
\[\phi_{ij}:F_iF_j\rightarrow F_j F_i, i,j=1,2.\]
We assume that $\phi_{ii}=\epsilon_i \id_{F_iF_i}$ and that
$\phi_{ij}$ is the inverse of $\phi_{ji}$ for $i,j=1,2$. Let $F=F_1
F_2$. Then the above commutation morphisms yield two
$(F,F)$-equivariant functors $(F_1,
F_1\phi_{12}:F_1F_1F_2\rightarrow F_1F_2F_1)$ and $(F_2,
\phi_{21}F_2:F_2F_1F_2\rightarrow F_1F_2F_2)$. It follows
from~\cite[Section 2.2]{Keller08d} that the induced functors
$\bar{F}_1$ and $\bar{F}_2$ satisfies: $\bar{F}_1\bar{F}_2\cong
\Delta(\epsilon_1\epsilon_2\id_F)$.

\subsection{The universal property of triangulated orbit
categories}\label{ss:universal-property}


Let us be given a triangle auto-equivalence $F:\cc\rightarrow\cc$, a triangle functor $\Phi:\cc\rightarrow\cc'$ and a natural isomorphism $\Phi\simeq \Phi\circ F$. Assume that they all admit dg lifts, and that the orbit category $\cc/F$ admits a canonical triangle structure. Then by~\cite[Section
9.4]{Keller05}, the triangle functor
 $\Phi$ induces a
triangle functor
\[\bar{\Phi}:\cc/F\longrightarrow \cc'\]
with a natural isomorphism $\bar{\Phi}\simeq\bar{\Phi}\circ \pi_\cc$.

\begin{lemma}\label{l:universal-property} Keep the above notations and assumptions. \begin{itemize}
\item[(a)] If $\Phi$ is essentially surjective, so is $\bar{\Phi}$.
\item[(b)] If $\Phi$ induces bijections for any objects $X$ and $Y$ of $\cc$
\[\bigoplus_{p\in\Z}\Hom_{\cc}(X,F^p Y)\longrightarrow \Hom_{\cc'}(\Phi X,\Phi
Y),\] then the functor $\bar{\Phi}$ is fully faithful.
\end{itemize}
\end{lemma}
\begin{proof} (a) This is because on objects $\bar{\Phi}$ takes the same value as
$\Phi$.

(b) For objects $X$ and $Y$ of $\cc$ and for an integer $p$, the
functor $\Phi$ gives us a map
\[\Hom_{\cc}(X,F^p Y)\stackrel{\Phi(X,F^p Y)}{\longrightarrow} \Hom_{\cc'}(\Phi X,\Phi F^p Y)\stackrel{\sim}{\longrightarrow}\Hom_{\cc'}(\Phi X,\Phi  Y).\]
Summing them up for all integers $p$ yields a map
\[\bigoplus_{p\in\Z}\Hom_{\cc}(X,F^p Y)\longrightarrow \Hom_{\cc'}(\Phi X,\Phi
Y),\]  which is precisely the map $\bar{\Phi}(X,Y)$. Therefore,
$\bar{\Phi}$ is fully faithful under the assumption.
\end{proof}

\begin{remark} In this appendix, we will apply
Lemma~\ref{l:universal-property} without checking the existence of
the required dg lift: this is standard in all cases. \end{remark}

\subsection{Hereditary graded algebras}\label{ss:hereditary-algebra}

Let $\Gamma$ be a hereditary graded $k$-algebra. Let
$\Grmod(\Gamma)$ be the category of graded modules over the graded
algebra $\Gamma$ and $\grmod(\Gamma)$ be its subcategory of
finite-dimensional graded modules. Let $\langle 1\rangle$ be the
degree shifting functor of $\Grmod(\Gamma)$. We grade our algebras
and modules cohomologically, so for a graded $\Gamma$-module
$M=\bigoplus_{p\in\mathbb{Z}} M^p$ the degree shifting $M\langle
1\rangle$ is defined by $(M\langle 1\rangle)^p=M^{p+1}$. Consider
$\Gamma$ as a dg $k$-algebra with trivial differential, and let
$\cd(\Gamma)$ and $\cd_{fd}(\Gamma)$ respectively denote the derived
category and the finite-dimensional derived category. The two
categories $\cd(\Gamma)$ and $\Grmod(\Gamma)$ are closely related.

\begin{lemma}\label{l:H^*}
 The taking total cohomology functor $H^*:\cd(\Gamma)\rightarrow \Grmod(\Gamma)$
 induces a bijection from the isoclasses of
 indecomposable objects of $\cd(\Gamma)$ to those of $\Grmod(\Gamma)$ and satisfies $H^*\circ\Sigma=\langle 1\rangle\circ H^*$.
\end{lemma}
\begin{proof}
Apply \cite[Theorem 3.1]{KellerYangZhou09} to the triangulated
category $\cd(\Gamma)$ and its compact generator $\Gamma$.
\end{proof}

A complex of graded $\Gamma$-modules can be viewed as a bicomplex.
Taking total complex induces a triangle functor from
$\cd(\Grmod(\Gamma))$ to $\cd(\Gamma)$ which restricts to a triangle
functor from $\cd^b(\grmod(\Gamma))$ to $\cd_{fd}(\Gamma)$. Let us
denote it by $\Tot$.

\begin{lemma}\label{l:tot}
\begin{itemize}
\item[(a)] The functor $\Tot$ is essentially surjective.
\item[(b)] We have a natural isomorphism of triangle functors $\Tot\circ\Sigma\circ \langle -1\rangle\cong\Tot$.
\item[(c)] For two objects $X$ and $Y$ of $\cd(\Grmod(\Gamma))$, we have a
bijection induced by $\Tot$
\[\bigoplus_{p\in\mathbb{Z}} \Hom_{\cd(\Grmod(\Gamma))}(X,\Sigma^p Y\langle -p\rangle)\longrightarrow \Hom_{\cd(\Gamma)}(\Tot X,\Tot Y).\]
\end{itemize}
\end{lemma}

\begin{proof} The statement (a) is a consequence of
Lemma~\ref{l:H^*} and the statement (b) follows immediately from the
definition of $\Tot$. The existence of the bifunctorial map
\[\bigoplus_{p\in\mathbb{Z}} \Hom_{\cd(\Grmod(\Gamma))}(X,\Sigma^p
Y\langle-p\rangle)\longrightarrow \Hom_{\cd(\Gamma)}(\Tot X,\Tot
Y)\] was shown in the proof of Lemma~\ref{l:universal-property}. It
remains to prove the bijectivity. Since $\Tot$ commutes with
infinite direct sums, by infinite d\'{e}vissage it suffices to prove
this for $X=\Gamma$ and $Y=\Sigma^q \Gamma\langle q'\rangle$ for all
integers $q$ and $q'$. We have
\begin{eqnarray*}
LHS&=&\bigoplus_{p\in\mathbb{Z}}
\Hom_{\cd(\Grmod(\Gamma))}(\Gamma,\Sigma^p \Sigma^q\Gamma\langle
q'\rangle\langle-p\rangle)\\
&=&\Hom_{\cd(\Grmod(\Gamma))}(\Gamma,\Gamma\langle q+q'\rangle)\\
&\cong&H^0\Gamma\langle q+q'\rangle\\
&\cong&\Hom_{\cd(\Gamma)}(\Gamma,\Sigma^{q+q'}\Gamma)\\
&=&RHS.
\end{eqnarray*}
\end{proof}

The functor $\Sigma\circ\langle -1\rangle$ is a triangle
auto-equivalence of $\cd^b(\grmod(\Gamma))$. It satisfies the
conditions in Theorem~\ref{t:keller's-theorem}, and hence the orbit
category $\cd^b(\grmod(\Gamma))/\Sigma\circ\langle-1\rangle$ admits
a canonical triangle structure.

\begin{proposition}\label{p:finite-dimensional-derived-category-as-orbit-category} The orbit category
$\cd^b(\grmod(\Gamma))/\Sigma\circ\langle-1\rangle$ is triangle
equivalent to $\cd_{fd}(\Gamma)$.
\end{proposition}
\begin{proof} This follows from Lemma~\ref{l:tot} and
Lemma~\ref{l:universal-property}.
\end{proof}


\subsection{Graded modules over $\Gamma$ and quiver representations}

Let $\overrightarrow{A}^{\infty}_{\infty}$ be the following quiver
of type $A^{\infty}_{\infty}$
\[\xymatrix{\ldots   & i-1 \ar[l] & i \ar[l] & i+1\ar[l] &\ldots\ar[l]~. }\]
Let $Q=Q^d$ be the disjoint union of $|d|$-copies of
$\overrightarrow{A}^{\infty}_{\infty}$, whose vertices are labeled
$(j,i)$, $0\leq j \leq |d|-1$, $i\in\mathbb{Z}$. Let $\sigma$ be the
unique automorphism of $Q$ which takes the following values on
vertices
\begin{eqnarray*}
\sigma(j,i)&=&(j-1-\lfloor\frac{j-1}{|d|}\rfloor|d|,i+sgn(d)\lfloor\frac{j-1}{|d|}\rfloor)\\
&=&\begin{cases} (j-1,i) & \text{if } 1\leq j\leq |d|-1\\
(|d|-1,i+sgn(d)) & \text{if } j=0~,\end{cases}
\end{eqnarray*}
where $sgn(d)$ is the sign of $d$, and $\lfloor x\rfloor$ is the
greatest integer smaller than or equal to $x$. Pushing out along the
automorphism $\sigma$ is an auto-equivalence of $\Rep(Q)$, still
denoted by $\sigma$. For a finite-dimensional representation $M$
written as a tuple $M=(M_0,M_1,\ldots,M_{|d|-1})$, we have
$\sigma(M)=(M_1,M_2,\ldots,M_{|d|-1},\tau^{-sgn(d)}M_0)$, where
$\tau$ is the Auslander--Reiten translation of
$\rep(\overrightarrow{A}^{\infty}_{\infty})$. The following is an
easy observation.

\begin{lemma} There is an equivalence of categories between
$\Grmod(\Gamma)$ and $\Rep(Q)$ such that it restricts to an
equivalence between $\grmod(\Gamma)$ and $\rep(Q)$ and the following
diagrams are commutative
\[\xymatrix{\Grmod(\Gamma)\ar[r]^{\sim}\ar[d]^{\langle 1\rangle}&\Rep(Q)\ar[d]^{\sigma}\\
\Grmod(\Gamma)\ar[r]^{\sim}&\Rep(Q)}\qquad\qquad
\xymatrix{\grmod(\Gamma)\ar[r]^{\sim}\ar[d]^{\langle 1\rangle}&\rep(Q)\ar[d]^{\sigma}\\
\grmod(\Gamma)\ar[r]^{\sim}&\rep(Q).}\]
\end{lemma}

\subsection{A covering functor}
Let $Q$ and $\sigma$ be as in the preceding subsection. Recall that
$n$ is a positive integer. Let $m$ be the greatest common divisor of
$n$ and $|d|$, and let $n'=\frac{n}{m}$, $d'=\frac{d}{m}$. Let $c$
be the inverse of $d'$ modulo $n'$.

Let $\bar{Q}$ be the quotient quiver of $Q$ under the automorphism
$\sigma^n$. Precisely, $\bar{Q}$ is the disjoint union of $m$-copies
of $\overrightarrow{\Delta}_{n'}$, where
$\overrightarrow{\Delta}_{n'}$ is the cyclic quiver with $n'$
vertices, \ie the quiver
\[\xymatrix@!=0.5pt{& 1\ar[dl]&&2\ar[ll]&\\
0\ar[dr]&&&& \cdot\ar[ul]\\
&n'-1\ar[rr]&&n'-2\ar@{.}[ur] &}\] 
The vertices of $\bar{Q}$ are labeled $(j,i)$, $0\leq j\leq m-1$,
$0\leq i\leq n'-1$. Moreover, the covering map
$C:Q\rightarrow\bar{Q}$ is given by the unique map between quivers
which takes the the following value on vertices
\[C(j,i)=(j-\lfloor\frac{j}{m}\rfloor m, i-\lfloor\frac{j}{m}\rfloor c-\lfloor\frac{i-\lfloor\frac{j}{m}\rfloor c}{n'}\rfloor n').\]
The map $C$ induces a pair of adjoint triangle functors
\[\xymatrix{\cd(\Rep Q)\ar@<.7ex>[r]^{C_*}&\cd(\Rep \bar{Q})\ar@<.7ex>[l]^{C^*},}\]
where $C^*$ is the pull-back functor and $C_*$ is the push-out
functor:
\[C_*(X)_{(\overline{i},\overline{j})}=\bigoplus_{C(i,j)=(\overline{i},\overline{j})}X_{(i,j)},~~
\text{for}~~ X\in\cd(\Rep
Q)~~\text{and}~~(\overline{i},\overline{j})\in\bar{Q}_0.\] The
functor $C_*$ is essentially surjective, and restricts to a triangle
functor $\cd^b(\rep Q)\rightarrow\cd^b(\rep \bar{Q})$ (here by
$\rep$ we mean the category of finite-dimensional nilpotent
representations). By abuse of notation, we denote $C=C_*$. It is
easy to prove

\begin{lemma}\label{l:orbit-map} Let $X$ be an object of $\cd(\Rep Q)$. Then
\[C^*C(X)=\bigoplus_{p\in\mathbb{Z}}\sigma^{np}(X).\]
\end{lemma}

 Let $\overline{\sigma}:\bar{Q}\rightarrow\bar{Q}$
be the unique automorphism of $\bar{Q}$ taking the following value
on vertices
\begin{eqnarray*}\bar{\sigma}(j,i)&=&(j-1-\lfloor\frac{j-1}{m}\rfloor
m,i-\lfloor\frac{j-1}{m}\rfloor
c-\lfloor\frac{i-\lfloor\frac{j-1}{m}\rfloor c}{n'}\rfloor n')\\
&=&\begin{cases} (j-1,i) & \text{ if } 1\leq j\leq m-1\\
(m-1,i+c-\lfloor\frac{i+c}{n'}\rfloor n') & \text{ if }
j=0.\end{cases}\end{eqnarray*} Then the following diagram is
commutative
\[\xymatrix{Q\ar[r]^C\ar[d]^{\sigma^{-1}}&\bar{Q}\ar[d]^{\overline{\sigma}^{\hspace{1pt}-1}}\\
Q\ar[r]^C&\bar{Q},}\] and induces a commutative diagram of triangle
functors
\[\xymatrix{\cd^b(\grmod\Gamma)\ar[d]^{\Sigma\circ\langle-1\rangle}\ar[r]^{\sim}&\cd^b(\rep Q)\ar[d]^{\Sigma\circ\sigma^{-1}}
\ar[r]^C&\cd^b(\rep \bar{Q})\ar[d]^{\Sigma\circ\overline{\sigma}^{\hspace{1pt}-1}}\\
\cd^b(\grmod \Gamma)\ar[r]^{\sim}&\cd^b(\rep
Q)\ar[r]^C&\cd^b(\rep\bar{Q}).}\] Here, by abuse of notation, we
denote by $\overline{\sigma}$ the push-out functor along the
automorphism $\overline{\sigma}$. By
Theorem~\ref{t:keller's-theorem}, the orbit category $\cd^b(\rep
\bar{Q})/\Sigma\circ\bar{\sigma}^{\hspace{1pt}-1}$ admits a
canonical triangle structure. Thus by
Lemma~\ref{l:universal-property} and
Proposition~\ref{p:finite-dimensional-derived-category-as-orbit-category},
we obtain a chain of triangle functors
\[\cd_{fd}(\Gamma)\stackrel{\sim}{\longrightarrow}\cd^b(\grmod\Gamma)/\Sigma\circ\langle-1\rangle\stackrel{\sim}{\longrightarrow}
\cd^b(\rep
Q)/\Sigma\circ\sigma^{-1}\stackrel{\bar{C}}{\longrightarrow}\cd^b(\rep
\bar{Q})/\Sigma\circ\bar{\sigma}^{\hspace{1pt}-1}.\] Let $\Phi$ be
the composition of the above three triangle functors.

\begin{lemma}\label{l:bar-phi} Assume that $n$ is even.
\begin{itemize}
\item[(a)] We have a natural isomorphism of triangle functors $\Phi\circ\Sigma^n\simeq\Phi$.
\item[(b)] The functor $\Phi$ is essentially surjective.
\item[(c)] The triangle functor $\Phi$ induces a bijection for any objects $M$ and $N$ of $\cd_{fd}(\Gamma)$
\[\bigoplus_{p\in\mathbb{Z}}\Hom_{\cd_{fd}(\Gamma)}(M,\Sigma^{pn}N)\longrightarrow\Hom_{\cd^b(\rep\bar{Q})/\Sigma\circ\bar{\sigma}^{\hspace{1pt}-1}}(\Phi(M),\Phi(N)).\]\end{itemize}
\end{lemma}

\begin{proof}
(a) The triangle structures of $\Sigma$ and $\bar{\sigma}$ yield the
following commutation morphisms
\[-\id_{\Sigma^2}:\Sigma\circ\Sigma\rightarrow\Sigma\circ\Sigma,\qquad\id_{\bar{\sigma}^{-2}}:\bar{\sigma}^{\hspace{1pt}-1}\circ\bar{\sigma}^{\hspace{1pt}-1}\rightarrow\bar{\sigma}^{\hspace{1pt}-1}\circ\bar{\sigma}^{\hspace{1pt}-1},\]
\[\phi_{12}:\Sigma\circ\bar{\sigma}^{\hspace{1pt}-1}\rightarrow\bar{\sigma}^{\hspace{1pt}-1}\circ\Sigma,\qquad \phi_{21}:\bar{\sigma}^{\hspace{1pt}-1}\circ\Sigma\rightarrow\Sigma\circ\bar{\sigma}^{\hspace{1pt}-1}.\]
By Section~\ref{ss:induced-functor}, we have two induced (triangle)
auto-equivalences of
$\cd^b(\rep\bar{Q})/\Sigma\circ\bar{\sigma}^{\hspace{1pt}-1}$,
which, by abuse of notation, will still be denoted by $\Sigma$ and
$\bar{\sigma}^{-1}$. Moreover,
$\Sigma\circ\bar{\sigma}^{-1}\cong\Delta(-\id_{\Sigma\circ\bar{\sigma}^{\hspace{1pt}-1}})$.

Now since $\Phi$ is a triangle functor, it follows that
\[\Phi\circ\Sigma^n\cong \Sigma^n\circ\Phi\cong
\Delta((-1)^n\id_{\Sigma\circ\bar{\sigma}^{\hspace{1pt}-1}})\circ\bar{\sigma}^n\circ\Phi\cong\Phi.\]

(b) In view of Lemma~\ref{l:universal-property}, this is because $C$
is essentially surjective.

(c) It suffices to prove that the functor $C$ induces a bijection
for any object $X$ and $Y$ of $\cd^b(\rep Q)$
\[\bigoplus_{p\in\mathbb{Z}}\Hom_{\cd^b(\rep Q)}(X,\sigma^{pn}Y)\longrightarrow\Hom_{\cd^b(\rep\bar{Q})}(C(X),C(Y)).\]
By Lemma~\ref{l:orbit-map}, the space on the left is isomorphic to
$\Hom_{\cd^b(\rep Q)}(X,C^*C(Y))$. Thus the bijectivity of the map
under investigation follows from the adjointness of $C$ and $C^*$.
\end{proof}

\begin{proposition}\label{p:characterization}  Let $n\in\mathbb{N}$ be even. We have a triangle equivalence
\[\bar{\Phi}:\cd_{fd}(\Gamma)/\Sigma^n\longrightarrow\cd^b(\rep\bar{Q})/\Sigma\circ\bar{\sigma}^{\hspace{1pt}-1}.\]
\end{proposition}
\begin{proof} This follows from Lemma~\ref{l:universal-property}
and Lemma~\ref{l:bar-phi}.
\end{proof}

\subsection{The characterization}
Recall that $m$ is the greatest common divisor of $n$ and $|d|$,
$n'=\frac{n}{m}$, $d'=\frac{d}{m}$, and $c$ is the inverse of $d'$
modulo $n'$.

Let $\tau$ be the Auslander--Reiten translation of
$\cd^b(\rep\overrightarrow{\Delta}_{n'})$. It follows from
Theorem~\ref{t:keller's-theorem} that the orbit category
$\cd^b(\rep\overrightarrow{\Delta}_{n'})/\tau^c\circ\Sigma^m$ admits
a canonical triangle structure. Let $\pi:\cd^b(\rep\overrightarrow{\Delta}_{n'})\rightarrow\cd^b(\rep\overrightarrow{\Delta}_{n'})/\tau^c\circ\Sigma^m$ denote the canonical projection functor. Recall that $\bar{Q}$ is the disjoint union of
$m$-copies of $\overrightarrow{\Delta}_{n'}$. So an object $X$ of $\cd^b(\rep\bar{Q})$ can be written as an ordered sequence $X=(X_0,\ldots,X_{m-1})$, where $X_0,\ldots,X_{m-1}\in\cd^b(\rep\overrightarrow{\Delta}_{n'})$. We define a triangle functor $\Pi:\cd^b(\rep\bar{Q})\rightarrow\cd^b(\rep\overrightarrow{\Delta}_{n'})$ by setting $\Pi(X)=\bigoplus_{j=0}^{m-1}\Sigma^{m-1-j}X_j$. Let $\Psi=\pi\circ\Pi$ be the composition.

\begin{lemma}\label{l:simplification}
The functor $\Psi:\cd^b(\rep\bar{Q})\rightarrow\cd^b(\rep\overrightarrow{\Delta}_{n'})/\tau^c\circ\Sigma^m$ induces a triangle equivalence
\[\bar{\Psi}:\cd^b(\rep\bar{Q})/\Sigma\circ\bar{\sigma}^{\hspace{1pt}-1}\longrightarrow\cd^b(\rep\overrightarrow{\Delta}_{n'})/\tau^c\circ\Sigma^m.\]
\end{lemma}
\begin{proof} We claim that the following diagram is commutative
\[\xymatrix{
\cd^b(\rep\bar{Q})\ar[r]^(0.35){{\Psi}}\ar[d]^{\Sigma\circ\bar{\sigma}^{\hspace{1pt}-1}} & \cd^b(\rep \overrightarrow{\Delta}_{n'})/\tau^c\circ\Sigma^m\ar@{=}[d] &\\
\cd^b(\rep\bar{Q})\ar[r]^(0.35){{\Psi}} &\cd^b(\rep
\overrightarrow{\Delta}_{n'})/\tau^c\circ\Sigma^m.}\]
Indeed, for $X=(X_0,\ldots,X_{m-1})\in\cd^b(\rep\bar{Q})$, we have
\begin{eqnarray*}
\Psi(X)&=&\pi(\bigoplus_{j=0}^{m-1}\Sigma^{m-1-j}X_j),\\
\Psi\circ\Sigma\circ\bar{\sigma}^{\hspace{1pt}-1}(X)&=&\Psi(\tau^c\circ\Sigma X_{m-1},\Sigma X_1,\ldots,\Sigma X_{m-2})\\
&=&\pi(\tau^c\circ\Sigma^m X_{m-1}\oplus\bigoplus_{j=0}^{m-2}\Sigma^{m-1-j}X_j).
\end{eqnarray*}
The natural isomorphism $\Psi\simeq\Psi\circ\Sigma\circ\bar{\sigma}^{\hspace{1pt}-1}$ is then induced from the natural isomorphism $\pi\simeq\pi\circ\tau^c\circ\Sigma^m$. It is clear that $\bar{\Psi}$ is essentially surjective because so is $\Psi$. Moreover, $\bar{\Psi}$ induces an identity of morphism spaces
\begin{eqnarray*}
\Hom_{\cd^b(\rep\overrightarrow{\Delta}_{n'})/\tau^c\circ\Sigma^m}(\bar\Psi X,\bar\Psi Y)&=&\bigoplus_{p\in\mathbb{Z}}\Hom_{\cd^b(\rep\bar{Q})}(X,(\Sigma\circ\bar{\sigma}^{\hspace{1pt}-1})^{p}Y).
\end{eqnarray*}
Applying
Lemma~\ref{l:universal-property} yields the desired result: $\bar{\Psi}$ is a triangle equivalence. It remains to prove the above identity. Writing $X=(X_0,\ldots,X_{m-1})$ and $Y=(Y_0,\ldots,Y_{m-1})$, it follows by induction that for $p\in\mathbb{Z}$
\[(\Sigma\circ\bar{\sigma}^{\hspace{1pt}-1})^pY=\Sigma^p(\tau^{cp'}Y_{m-j_p},\ldots,\tau^{cp'}Y_{m-1},\tau^{c(p'-1)}Y_0,\ldots,\tau^{c(p'-1)}Y_{m-j_p-1}),\]
where $p'=\lceil\frac{p}{m}\rceil$ is the smallest integer greater than equal to $\frac{p}{m}$, and $j_p=m+p-mp'$.
Therefore (below $(?,?)=\Hom_{\cd^b(\rep\overrightarrow{\Delta}_{n'})}(?,?)$)
\begin{eqnarray*}
RHS
&=&\bigoplus_{p\in\mathbb{Z}}(\bigoplus_{j=0}^{j_p-1}(X_j,\Sigma^p\circ\tau^{cp'}Y_{m-j_p+j})
\oplus\bigoplus_{j=j_p}^{m-1}(X_j,\Sigma^p\circ\tau^{c(p'-1)}Y_{-j_p+j}))\\
&=&\bigoplus_{p'\in\mathbb{Z}}\bigoplus_{j_p=1}^m(\bigoplus_{j=0}^{j_p-1}(X_j,\Sigma^{j_p-m+mp'}\circ\tau^{cp'}Y_{m-j_p+j})
\oplus\bigoplus_{j=j_p}^{m-1}(X_j,\Sigma^{j_p-m+mp'}\circ\tau^{c(p'-1)}Y_{-j_p+j}))\\
&=&\bigoplus_{p'\in\mathbb{Z}}\bigoplus_{j=0}^{m-1}(\bigoplus_{j_p=j+1}^m(X_j,\Sigma^{j_p-m+mp'}\circ\tau^{cp'}Y_{m-j_p+j})
\oplus\bigoplus_{j_p=1}^j(X_j,\Sigma^{j_p-m+mp'}\circ\tau^{c(p'-1)}Y_{-j_p+j}))\\
&=&\bigoplus_{p'\in\mathbb{Z}}\bigoplus_{j=0}^{m-1}(\bigoplus_{j'=j}^{m-1}(X_j,\Sigma^{j-j'+mp'}\circ\tau^{cp'}Y_{j'})
\oplus\bigoplus_{j'=0}^{j-1}(X_j,\Sigma^{j-j'+m(p'-1)}\circ\tau^{c(p'-1)}Y_{j'}))\\
&=&\bigoplus_{p'\in\mathbb{Z}}\bigoplus_{j=0}^{m-1}\bigoplus_{j'=0}^{m-1}(X_j,\Sigma^{j-j'+mp'}\circ\tau^{cp'}Y_{j'}).
\end{eqnarray*}
On the other hand, we have
\begin{eqnarray*}
LHS&=&\bigoplus_{p\in\mathbb{Z}}(\Psi X,(\tau^c\circ\Sigma^m)^{p}\Psi Y)\\
&=&\bigoplus_{p\in\mathbb{Z}}(\bigoplus_{j=0}^{m-1}\Sigma^{m-1-j}X_j,(\tau^c\circ\Sigma^m)^{p}\bigoplus_{j=0}^{m-1}\Sigma^{m-1-j}Y_j)\\
&=&\bigoplus_{p\in\mathbb{Z}}\bigoplus_{j=1}^{m-1}\bigoplus_{j'=0}^{m-1}(X_j,\tau^{cp}\circ\Sigma^{j-j'+mp}Y_{j'}).
\end{eqnarray*}
\end{proof}

Combining Proposition~\ref{p:characterization} and
Lemma~\ref{l:simplification}, we obtain the main result of this
section.

\begin{theorem}\label{t:characterization} Let $n\in\mathbb{N}$ be even. Then we have a triangle equivalence \[\bar{\Psi}\circ\bar{\Phi}:\cd_{fd}(\Gamma)/\Sigma^n\longrightarrow\cd^b(\rep\overrightarrow{\Delta}_{n'})/\tau^c\circ\Sigma^m.\]
\end{theorem}


\end{document}